\documentclass[11pt]{amsart}

\providecommand{\part}[1]{%
  \clearpage
  \section*{#1}%
  \addcontentsline{toc}{section}{#1}%
}

\usepackage{amsmath,amssymb,amsfonts,amsthm,amscd,indentfirst}
\usepackage{latexsym,amsxtra,mathrsfs}
 
\usepackage[top=3cm, bottom=3.5cm, right=2.5cm, left=2.2cm]{geometry}

\usepackage{hyperref}
\usepackage{microtype}

\usepackage{bookmark}
\usepackage{amsmath,thmtools,mathtools}
\usepackage{etoolbox}

\allowdisplaybreaks 
\mathtoolsset{showonlyrefs=true}
\providecommand{\mathsection}{\mathord{\S}}
\DeclareRobustCommand{\textsection}{\ifmmode\mathsection\else\S\fi}

\usepackage[
  backend=biber,
  style=alphabetic,
  sorting=nyt,
  maxbibnames=99,
  giveninits=true,
]{biblatex}
\renewbibmacro{in:}{}

\DeclareFieldFormat[article,inproceedings,incollection,unpublished]{title}{#1}
\DeclareFieldFormat[article,inproceedings,incollection,unpublished]{citetitle}{#1}

\DeclareSourcemap{%
  \maps[datatype=bibtex]{%
    \map{\step[fieldsource=shortjournal, fieldtarget=journaltitle]}%
  }%
}

\makeatletter
\newcommand*\bbx@lasthash{}
\newtoggle{bbx@dashed}
\AtBeginBibliography{\global\let\bbx@lasthash\empty}
\AtEveryBibitem{%
  \global\togglefalse{bbx@dashed}%
  \iffieldundef{fullhash}
    {\global\let\bbx@lasthash\empty}
    {\iffieldequalstr{fullhash}{\bbx@lasthash}{\global\toggletrue{bbx@dashed}}{}%
     \xdef\bbx@lasthash{\thefield{fullhash}}}%
}
\DeclareNameFormat{labelname}{%
  \ifboolexpr{togl {bbx@dashed}}
    {\ifnum\value{listcount}=1\relax\bibnamedash\fi}
    {\nameparts{#1}%
     \usebibmacro{name:family-given}
       {\namepartfamily}
       {\namepartgiven}
       {\namepartprefix}
       {\namepartsuffix}%
     \usebibmacro{name:andothers}}%
}
\makeatother

\newcounter{mparcnt}

\usepackage{fancyhdr}
\usepackage{esint}
\usepackage{enumerate}
\usepackage{xcolor}

\usepackage{pictexwd,dcpic}
\usepackage{graphicx}
\usepackage{caption}

\usepackage{graphicx}
\usepackage{caption}
\usepackage{slashed}

\usepackage{epsfig,here}
\usepackage{subfigure,here}

\newtheorem{theorem}{Theorem}[section]
\newtheorem{lemma}[theorem]{Lemma}
\newtheorem{proposition}[theorem]{Proposition}
\newtheorem{definition}[theorem]{Definition}
\newtheorem{corollary}[theorem]{Corollary}
\newtheorem{remark}[theorem]{Remark}

\newcommand{\abs}[1]{\lvert#1\rvert}
\newcommand{\Abs}[1]{\left\lvert#1\right\rvert}

\newcommand{\norm}[1]{\lVert#1\rVert}

\newcommand{\rd}{{\rm d}}
\newcommand{\rdV}{{\rm dV}}

\newcommand{\rid}{{\rm id}}

\newcommand{\D}{{\slashed{D}}}
\newcommand{\rRe}{{\rm Re}}
\newcommand{\rIm}{{\rm Im}}
\newcommand{\curl}{{\rm curl}\,}
\newcommand{\ip}{\lrcorner\,}
\newcommand{\w}{\wedge}
\newcommand{\p}{\partial}

\def\<{\langle}
\def\>{\rangle}
\def\S{\mathbb{S}}
\def\R{\mathbb{R}}

\newcommand{\eq}[1]{\begin{equation}\allowdisplaybreaks\begin{alignedat}{2} #1 \end{alignedat}\end{equation}}

\numberwithin{equation} {section}

\begin{document}

	
\title[Another sharp criterion for Dirac zero modes]
{
Another sharp criterion for Dirac zero modes
}
\date{\today}


\author{Guofang Wang}
\address{ Albert-Ludwigs-Universit\"at Freiburg,
Mathematisches Institut,
Ernst-Zermelo-Str. 1,
D-79104 Freiburg, Germany}
\email{guofang.wang@math.uni-freiburg.de}

\author{Mingwei Zhang}
\address{ Wuhan University, School of Mathematics and Statistics, 430072 Wuhan, China and 
Albert-Ludwigs-Universit\"at Freiburg,
Mathematisches Institut,
Ernst-Zermelo-Str. 1,
D-79104 Freiburg, Germany}
\email{zhangmwmath@whu.edu.cn}

\begin{abstract}
We resolve an open problem posed by Frank--Loss \cite{FL1}. Let $n\ge 3$ and let $\varphi\in L^p(\S^n)$, with $\frac{n}{n-1}<p<\infty$, be a nontrivial Dirac zero mode on $\S^n$, i.e. a nonzero spinor satisfying
\eq{
    \D\varphi = iA\cdot\varphi,
}
where $\D$ is the Dirac operator and $A$ is a vector field with $\rd A^\flat\in L^{n/2}$. We prove the sharp lower bound
\eq{
   \norm{\rd A^\flat}_{\frac{n}{2}} \ge 2\left[\frac{n}{2}\right]^{-\frac12}\frac{n-1}{n-2}S_n
   = \left[\frac{n}{2}\right]^{-\frac12}\frac{n(n-1)}{2}\omega_n^{\frac{2}{n}}.
}
Equality is attainable if and only if $n$ is odd; in that case, modulo conformal and gauge transformations, $\varphi$ is a Killing spinor and $A$ is a real multiple of the Reeb field associated with $\varphi$. The case $n=3$ was proved in our recent paper \cite{WZ26curl} using a different method. In the paper we divide the remaining cases into 3 cases: i) $n\ge 5$ is odd, ii) $n\ge 5$ is even and iii)  $n=4$. All these cases need to use different methods.

For $n\ge 5$, the argument crucally reduces to an improved Sobolev inequality on $\S^n$ under a barycenter constraint. Specifically, for $u\in W^{1,2}(\S^n)$ we consider
\eq{
    \mathfrak{a}_n \coloneqq \inf\Bigg\{ \frac{ \int \Big(\abs{\nabla u}^2 + \frac{n(n-2)}{4}u^2\Big) - \frac{n(n-2)}{4}\omega_n^{\frac{2}{n}}\norm{u}_{\frac{2n}{n-2}}^2 }{ \omega_n^{\frac{2}{n}}\norm{u}_{\frac{2n}{n-2}}^2 - \int u^2 } \,\Bigg|\, \int x\abs{u}^{\frac{2n}{n-2}}=0,\ \omega_n^{\frac{2}{n}}\norm{u}_{\frac{2n}{n-2}}^2 - \int u^2>0 \Bigg\},
}
and we obtain the following universal estimate
\eq{
    \mathfrak{a}_n > \frac{n(n-2)}{4(n^2-3n+1)},
}
which is enough for our aim. 
Determining the exact value of $\mathfrak{a}_n$ remains open.


\end{abstract}

\maketitle
\tableofcontents

\section{Introduction}

Zero modes of Dirac-type operators in the presence of magnetic fields play a central role in spectral theory and mathematical physics.  In $3$-dimensional case, they appear as obstructions in fermionic determinants and in semiclassical asymptotics, and they are the key mechanism behind instability phenomena for Coulomb systems coupled to magnetic fields. A quantitative way to formulate the size of a magnetic field is through conformally invariant norms.

Motivated by this viewpoint, Frank--Loss \cite{FL1} studied the Dirac zero mode equation on
$\R^3$
\eq{\label{eq:intro_R3}
    \sigma\cdot(-i\nabla-A)\varphi=0,
}
where $i$ is the imaginary unit, $\varphi$ is a spinor field, and $A$ is a vector field. They made a breakthrough by proving that if there exists a non-trivial solution, then $\norm{\curl A}_{L^{3/2}(\R^3)}\geq 2S_3$, where $S_3$ is the sharp critical Sobolev constant. It  a criterion for the existence of zero modes, though the lower bound is not sharp. 

To determine the sharp criterion is important for at least two reasons. First, in physics, $\curl A$ arises as the magnetic field in the magnetic Pauli models in the one-electron problem with spin. A classical variational argument (going back to Fr\"ohlich--Lieb--Loss \cite{Froelich-Lieb-Loss} and further developed by \cite{FL1}) shows that a universal lower bound of $\norm{\curl A}_{L^{3/2}}$ translates into an explicit quantitative obstruction to collapse, and a better lower bound leads to a better stability threshold.
Second, in geometry, $\curl A$ is the curvature of the connection $\rd - iA$ on the trivial complex line bundle. It is gauge invariant, hence it is interesting to seek an optimal lower bound. 

The present paper wants to find an optimal lower bound. Due to the conformal invariance of the zero mode equation, equation on $\R^n$ can be equivalently transferred to equation on $\S^n$ via stereographic
projection. Hence we focus in the paper on $\S^n$.
Our main goal is to establish the sharp lower bound for $\|\rd A^\flat\|_{L^{n/2}(\S^n)}$ under the
assumption that the zero mode equation admits a non-trivial solution, where $A^\flat$ is the dual $1$-form of $A$. We remark that $\norm{\rd A^\flat}_{L^{n/2}}$ is the higher dimensional version of $\norm{\curl A}_{L^{3/2}}$, and is  conformally invariant. 

For comparison, we recall the higher-dimensional  estimate of Frank--Loss \cite[Theorem~1.3]{FL1}. Let $n\ge 3$ and set $\nu\coloneqq [n/2]$. Then the spinor bundle has complex rank $2^\nu$. Assume that $\varphi\in L^p$ ($\frac{n}{n-1}<p<\infty$) is a nontrivial weak solution of
\eq{\label{eq:intro_FL13}
    \sigma\cdot(-i\nabla-A)\varphi=0,
}
where $\sigma=(\sigma_1,\dots,\sigma_n)$ is a system of higher-order Pauli matrices, standing for the representative of the Clifford algebra. It is equivalent to
\eq{\label{eq1.0}
    \D\varphi = iA\cdot\varphi,
}
where $\D$ is the Dirac operator. Let $\abs{\rd A^\flat}$ be the form length, namely 
\eq{\label{eq:intro_FL13_B}
    \abs{\rd A^\flat} = \Big(\sum_{j<k}\abs{\p_jA_k-\p_k A_j}^2\Big)^{\frac{1}{2}}.
}
Then
\eq{\label{eq:intro_FL13_ineq}
    \norm{\rd A^\flat}_{L^{\frac{n}{2}} } = \Big(\int_{\R^n}|\rd A^\flat|^{\frac{n}{2}}\,\rd x\Big)^{\!\frac{2}{n}}
    \ge \nu^{-\frac12}\,\frac{n-1}{n-2}\,S_n,
}
where $S_n=\frac{n(n-2)}{4}\abs{\S^n}^{\frac{2}{n}}$ is the optimal Sobolev constant, and $A^\flat $ is the dual $1$-form of $A$. 

We now state our main result. 

\begin{theorem}\label{thm_main}
Let $\varphi\in L^p$ ($\frac{n}{n-1}<p<\infty$) be a non-trivial zero mode on $\S^n$ ($n\geq3$) solving \eqref{eq1.0}, and  $B=\rd A^\flat\in L^{n/2}$ be a magnetic field. Then
\eq{\label{eq:main_ineq}
   \norm{\rd A^\flat}_{\frac{n}{2}} \geq 2 \nu^{-\frac 12} \frac{n-1}{n-2}S_n =\nu^{-\frac 12} \frac{n(n-1)}{2}\omega_n^{\frac{2}{n}}, \qquad \omega_n\coloneqq\abs{\S^n},
}
with equality if and only if $n$ is odd, and modulo conformal and gauge transformations, $\varphi$ is a Killing spinor and $A$ is a (real) multiple of the Reeb field associated with $\varphi$.
\end{theorem}

Theorem \ref{thm_main} gives an affirmative answer to an open problem asked by Frank--Loss \cite{FL1}.
The $n=3$ case was proved in our previous work \cite{WZ26curl}, by using the sharp curl--Sobolev inequality, together with a peculiar phenomenon of $3$-dimensional case. In this paper, we prove it for the general case $n\geq4$.
 For $n=4$, we exploit the $4$-dimensional Clifford structure to obtain a finer estimate, which leads to an even better lower bound.
 For $n\ge 5$, the first main step is using  a  refined decomposition  of $\nabla\varphi$ other than the one used in \cite{WZ26curl} to obtain the following estimates.

\begin{proposition}\label{prop_key0} 
Under the same assumptions as in Theorem \ref{thm_main}, we have for $u\coloneqq\abs{\varphi}^{\frac{n-2}{n-1}}$

\noindent(1) if $n$ is odd, then
\eq{\label{eq:key_odd0}
    \Big[\frac{n}{2}\Big]^{\frac{1}{2}} \norm{\rd A^\flat}_{\frac{n}{2}} \norm{u}_{\frac{2n}{n-2}}^2\geq 
   \int  u^2 \<\rd A^\flat,\beta\>
    \ge   \frac{2(n^2-3n+1)}{(n-2)^2}\int \abs{\nabla u}^2+ \frac{n(n-1)}{2}\int u^2;
}

\noindent(2) if $n$ is even, then
\eq{\label{eq:key_even0}
    \Big[\frac{n}{2}\Big]^{\frac{1}{2}} \norm{\rd A^\flat}_{\frac{n}{2}} \norm{u}_{\frac{2n}{n-2}}^2\geq 
  \int   u^2 \<\rd A^\flat,\beta\> \ge \frac{2(n^2-3n+1)}{(n-2)^2}\int \abs{\nabla u}^2 + \Big(\frac{n(n-1)}{2}-\frac{n}{4(n-1)}\Big)\int u^2.
}
Here $\beta$ is a $2$-form induced by $\varphi$, see Definition~\ref{def3.2}. \end{proposition}

The second main step is the following improved Sobolev inequality under barycenter constraint.

\begin{theorem}\label{thm_a_n0}
Let $n\geq5$. For $u\in W^{1,2}(\S^n)$ we set
\eq{\label{eq:a_n0}
    \mathfrak{a}_n \coloneqq \inf\Bigg\{ \frac{ \int \big(\abs{\nabla u}^2 + \frac{n(n-2)}{4}u^2\big) - \frac{n(n-2)}{4}\omega_n^{\frac{2}{n}}\norm{u}_{\frac{2n}{n-2}}^2 }{ \omega_n^{\frac{2}{n}}\norm{u}_{\frac{2n}{n-2}}^2 - \int u^2 } \,\Bigg|\, \int x\abs{u}^{\frac{2n}{n-2}}=0, \ \omega_n^{\frac{2}{n}}\norm{u}_{\frac{2n}{n-2}}^2 - \int u^2>0 \Bigg\}.
}
Then
\eq{\label{eq:goal0}
    \mathfrak{a}_n > \frac{n(n-2)}{4(n^2-3n+1)}.
}
\end{theorem}

Theorem \ref{thm_main} for odd $n\ge5$
follows then from \eqref{eq:key_odd0} and Theorem \ref{thm_a_n0}.  However, when $n\ge 6$ is even,  we only have a weaker estimate \eqref{eq:key_even0}, which is not enough. In order to deal with this problem, we use  the even dimensional chirality splitting of the Clifford algebra to improve  
\eqref{eq:key_even0}, see Subsection~\ref{sec7.2}. It is interesting to see that the  quantitative stability result in \cite{DEFFL22} implies Theorem \ref{thm_main} for sufficiently large dimension $n$, without using Theorem \ref{thm_a_n0}.

Since Theorem \ref{thm_a_n0} has its own interest,  we leave the detailed discussion and the proof in Part~\ref{part2}.

It worth to mention that Frank--Loss \cite{Frank_Loss_2024} provided a sharp criterion for the existence of zero modes in terms of another conformally invariant norm by
proving
\eq{
    \norm{A}_{L^n(\S^n)}^2 \geq \frac{n}{n-2}S_n = \frac{n^2}{4}\omega_n^{\frac{2}{n}}.
 }
 with equality if and only if $n$ is odd. They also classified the odd-dimensional equality case. Their results were generalized to general spin manifolds by \cite{Reuss25} and \cite{WZ25b}, together with classification of the equality case. Theorem \ref{thm_main} provides another sharp criterion in $\S^n$. It is interesting to ask if one can generalize Theorem \ref{thm_main} to general closed spin manifolds.

\

\noindent\textit{Organization of the paper.} 
We divide this paper into two parts. Part~\ref{part1} gives another sharp criterion for zero modes, in terms of $\norm{\rd A^\flat}_{n/2}$. Section~\ref{sec2} reviews the zero mode equation and its conformal and gauge invariance. Section~\ref{sec3} develops an argument of orthogonal decomposition, which is the key step in our proof; and constructs an important $2$-form induced by the spinor, which helps us to estimate the norm of the magnetic field. Section~\ref{sec4} gives a first estimate, showing that the main result is true for sufficiently large dimensions, as an application of the stability result \cite{DEFFL22}. Section~\ref{sec5} then completely proves it for any dimensions. The proof relies on an improved Aubin--Sobolev inequality, Theorem~\ref{thm_a_n0}, which is proved in Part~\ref{part2}. In Section~\ref{appendix_Clifford} as an appendix, we record the explicit Clifford structure of $\R^4$, which is needed in the proof of $n=4$ case.
Part 2 is devoted to the proof of Theorem \ref{thm_a_n0} and could be viewed as an independent work.

\part{Another sharp criterion for zero modes}\label{part1}

\section{Preliminaries}\label{sec2}

On the round sphere $\S^n$, we study spinors $\varphi$ solving the (magnetic) zero mode equation
\eq{\label{eq0:zero-mode}
    \D\varphi = iA\cdot\varphi,
}
where $\D$ is the Dirac operator, $A$ is a real vector field, and ``$\cdot$'' denotes Clifford multiplication.

It is convenient to absorb the potential into a unitary connection on the spinor bundle by setting
\eq{
    \nabla^A_X\psi \coloneqq \nabla_X\psi - i\<A,X\>\psi.
}
With respect to any local orthonormal frame $\{e_j\}$, the associated Dirac operator is
\eq{\label{eq:D^A}
    \D^A\psi \coloneqq e_j\cdot\nabla^A_{e_j}\psi = \D\psi - iA\cdot\psi.
}
Thus \eqref{eq0:zero-mode} is equivalent to
\eq{\label{eq:zero-mode}
    \D^A\varphi = 0.
}

Writing $f\coloneqq\abs{\varphi}$ and $\varphi=f\phi$ with $\abs{\phi}=1$ on $\{f>0\}$, we obtain from \eqref{eq:zero-mode} the identity
\eq{\label{eq:1}
    \D^A\phi = -\nabla(\log f)\cdot\phi.
}

We next record a commutator identity for $\nabla^A$.

\begin{lemma}\label{lem0.1}
At a given point, let $\{e_j\}$ be an orthonormal basis with $\nabla e_j=0$. For any spinor $\psi$ we have
\eq{
    \label{eq:lem2.1}[\nabla^A_{e_j},\nabla^A_{e_k}]\psi = -\frac{1}{2}e_j\cdot e_k\cdot\psi - i\rd A^\flat(e_j,e_k)\psi, \qquad j\neq k.
}
\end{lemma}
\begin{proof} For the proof, see \cite{Lawson}.
\end{proof}

We will also use the Schr\"odinger--Lichnerowicz identity for the twisted Dirac operator $\D^A$. Here we include a short proof for completeness.
\eqref{eq:S-L} indicates why one may obtain an estimate on $\rd A^\flat$ in terms of the Sobolev constant $S_n$.

\begin{lemma}
For any spinor $\psi$ we have
\eq{\label{eq:S-L}
    (\D^A)^2\psi = (\nabla^A)^*\nabla^A\psi + \frac{n(n-1)}{4}\psi - i\,\rd A^\flat\cdot\psi.
}
\end{lemma}
\begin{proof}
By definition \eqref{eq:D^A} we have
\eq{
    (\D^A)^2\psi &= e_j\cdot\nabla^A_{e_j}(e_k\cdot\nabla^A_{e_k}\psi) = e_j\cdot e_k\cdot\nabla^A_{e_j}\nabla^A_{e_k}\psi \\
    &= -\nabla^A_{e_j}\nabla^A_{e_j}\psi + \sum_{j\neq k}e_j\cdot e_k\cdot\nabla^A_{e_j}\nabla^A_{e_k}\psi \\
    &= (\nabla^A)^*\nabla^A\psi + \sum_{j<k}e_j\cdot e_k\cdot [\nabla^A_{e_j},\nabla^A_{e_k}]\psi.
}
Combining with Lemma~\ref{lem0.1} gives
\eq{
    (\D^A)^2\psi &= (\nabla^A)^*\nabla^A\psi - \frac{1}{2}\sum_{j<k}e_j\cdot e_k\cdot e_j\cdot e_k\cdot\psi - i\sum_{j<k}e_j\cdot e_k\cdot \rd A^\flat(e_j,e_k)\psi \\
    &= (\nabla^A)^*\nabla^A\psi + \frac{n(n-1)}{4}\psi - i\,\rd A^\flat\cdot\psi.
}
The identity follows.
\end{proof}

Let $\tilde g=e^{2\sigma}g$ be a conformal change of metric on $\S^n$. The spinor bundles $(\Sigma_g,\nabla)$ and $(\Sigma_{\tilde g},\tilde\nabla)$ are canonically identified by an isometry $F_\sigma:\Sigma_g\to\Sigma_{\tilde g}$, under which one has
\eq{
    \tilde{\nabla}_X(F_\sigma\psi) = F_\sigma\Big(\nabla_X\psi - \frac{1}{2}X\cdot\nabla\sigma\cdot\psi - \frac{1}{2}X(\sigma)\psi\Big),
}
and the conformal covariance of the Dirac operator
\eq{
    \tilde{\D}(e^{-\frac{n-1}{2}\sigma}F_\sigma\psi) = e^{-\frac{n+1}{2}\sigma}F_\sigma\D\psi.
}

Accordingly, if we define
\eq{
    \tilde{\varphi}\coloneqq e^{-\frac{n-1}{2}\sigma}F_\sigma\varphi, \qquad \tilde{A}\coloneqq e^{-2\sigma}A,
}
then
\eq{
    \tilde{A}^\flat=A^\flat, \qquad \abs{\rd A^{\flat}}_{\tilde g}=e^{-2\sigma}\abs{\rd A^{\flat}}_{g},
}
and
\eq{
    \tilde{f}=\abs{\tilde{\varphi}}_{\tilde{g}}=e^{-\frac{n-1}{2}\sigma}f, \qquad \tilde{\phi}=F_\sigma\phi.
}
In particular,
\eq{
    \D\varphi = iA\cdot_g\varphi \implies \tilde{\D}\tilde{\varphi} = i\tilde{A}\cdot_{\tilde{g}}\tilde{\varphi},
}
so $(\varphi,A,g)$ is a zero mode if and only if $(\tilde{\varphi},\tilde{A},\tilde{g})$ is.
Moreover, the natural conformal norms of the potential and curvature are invariant:
\eq{
    \norm{\tilde{A}}_{L^n(\tilde{g})} = \norm{A}_{L^n(g)}, \qquad \norm{\rd \tilde{A}^\flat}_{L^{\frac{n}{2}}(\tilde{g})} = \norm{\rd A^\flat}_{L^{\frac{n}{2}}(g)}.
}

Finally, \eqref{eq:zero-mode} is gauge invariant: for any real-valued function $h$,
\eq{
    (A,\varphi)\longmapsto (A+\nabla h,\,e^{ih}\varphi)
}
preserves solutions.

\section{An orthogonal decomposition and an associated 2-form}\label{sec3}

In this section we first decompose the covariant derivative $\nabla^A\varphi$ into orthogonal pieces adapted to the splitting $\varphi=f\phi$ from Section~2. Concretely, we regard $\nabla^A\varphi$ as a $T^*\S^n$-valued spinor, i.e. the field $X\mapsto \nabla^A_X\varphi$, and separate the contributions coming from $\nabla f$, from the ``model part'' forced by \eqref{eq:1}, and from a remainder term that measures the deviation from this model. An analogous decomposition for differential forms was used in \cite{WZ26curl}. Then we introduce an associated 2-form $\beta$ and establish the identities needed to estimate the magnetic field.

\subsection{Decomposition}
In this subsection we decompose $ \nabla ^A \varphi$ into orthogonal components.

We first decompose $\nabla^A \phi$ by
defining 
\eq{\label{eq:138}
    T_X\coloneqq \frac{1}{n-1}\Big( X\cdot \nabla(\log f)\cdot\phi + X(\log f)\phi \Big), \quad S_X\coloneqq \nabla^A_X\phi - T_X.
}
Then, for any vector field $X$, we decompose $\nabla \varphi$
\eq{
    \nabla^A_X\varphi = X(f)\phi + fT_X + fS_X \eqcolon P_X + Q_X + R_X.
}
It is clear that this decomposition is different from the usual decomposition in spin geometry $\nabla_X \varphi =(\nabla _X\varphi+\frac 1n X\cdot \D \ \varphi) -\frac 1n  X\cdot \D \varphi$, so that the first part, which is the so-called twistor operator,  is conformally invariant. 

\begin{remark}\label{rmk3.1}
Let $\tilde{g}=e^{2\sigma}g$. The discussion in Section~\ref{sec2} implies that 
\eq{
    \tilde{T}_{\tilde{X}} = e^{-\sigma}F_\sigma\Big( T_X-\frac{1}{2}X\cdot\nabla\sigma\cdot\phi-\frac{1}{2}X(\sigma)\phi \Big),
}
and
\eq{
    \tilde{S}_{\tilde{X}} = e^{-\sigma}F_\sigma(S_X),
}
where $\tilde{X}=e^{-\sigma}X$. Hence the remainder term $S$ is conformally covariant.
\end{remark}

\begin{lemma}\label{lem1.1}
$P,Q,R$ are pairwise orthogonal.
\end{lemma}
\begin{proof}
First,
\eq{
    \rRe\<P,T\> &= \frac{e_i(f)}{n-1}\<e_i\cdot\rd(\log f)\cdot\phi+e_i(\log f)\phi, \,\phi\> \\
    &= \frac{1}{n-1}\Big( \<\rd(\log f),\rd f\> - \<\rd(\log f),\rd f\> \Big) = 0,
}
hence
\eq{
    \rRe\<P,Q\> = f\rRe\<P,T\> = 0.
}
Second,
\eq{
    \rRe\<P,R\> = f\rRe\<P,S\> = f\rRe\<P,\nabla^A\phi-T\> = f\rRe\<P,\nabla^A\phi\> = fe_i(f)\rRe\<\phi, \nabla^A_{e_i}\phi \> = 0.
}
Third, using \eqref{eq:1} gives
\eq{\label{eq:2}
    \rRe\<\nabla^A\phi,T\> &= \frac{1}{n-1}\rRe\<\nabla^A_{e_i}\phi, \, e_i\cdot\rd(\log f)\cdot\phi+e_i(\log f)\phi\> \\
    &= -\frac{1}{n-1}\rRe\<\D^A\phi, \,\rd(\log f)\cdot\phi\> = \frac{1}{n-1}\abs{\nabla(\log f)}^2.
}
Moreover,
\eq{\label{eq:3}
     \rRe\<T,T\> &= \abs{T}^2 = \frac{1}{(n-1)^2} \Big( n\abs{\nabla(\log f)}^2 + \abs{\nabla(\log f)}^2 + 2\<e_i\cdot\rd(\log f)\cdot\phi,\, e_i(\log f)\phi\> \Big) \\
     &= \frac{1}{(n-1)^2} \Big( n\abs{\nabla(\log f)}^2 + \abs{\nabla(\log f)}^2 - 2\abs{\nabla(\log f)}^2 \Big) \\
     &= \frac{1}{n-1}\abs{\nabla(\log f)}^2.
}
Combining \eqref{eq:2} and \eqref{eq:3} yields
\eq{
    \rRe\<Q,R\> = f^2\rRe\<T,S\> = f^2\rRe\<T, \,\nabla^A\phi-T\> = 0.
}
Hence we complete the proof.
\end{proof}

\begin{lemma}\label{lem1.2}
We have
\eq{
    \abs{P}^2 = \abs{\nabla f}^2, \quad \abs{Q}^2 = \frac{1}{n-1}\abs{\nabla f}^2, \quad \abs{R}^2 = f^2\abs{S}^2.
}
\end{lemma}
\begin{proof}
The first and the third identities are trivial. For the second, we have by \eqref{eq:3}
\eq{
    \abs{Q}^2 = f^2\abs{T}^2 = \frac{1}{n-1}f^2\abs{\nabla(\log f)}^2 = \frac{1}{n-1}\abs{\nabla f}^2.
}
Hence we complete the proof.
\end{proof}

As a consequence of Lemma~\ref{lem1.1} and Lemma~\ref{lem1.2}, we have the following identity.
\begin{corollary}
We have
\eq{\label{eq:6}
    \abs{\nabla^A\varphi}^2 = \frac{n}{n-1}\abs{\nabla f}^2 + f^2\abs{S}^2.
}
\end{corollary}

\begin{lemma}\label{lem1.4}
We have
\begin{enumerate}
    \item $e_j\cdot P_{e_j}=\nabla f\cdot\phi$;
    \item $e_j\cdot Q_{e_j}=-\nabla f\cdot\phi$;
    \item $e_j\cdot R_{e_j}=0$.
\end{enumerate}
\end{lemma}
\begin{proof}
First,
\eq{
    e_j\cdot P_{e_j}=e_j\cdot e_j(f)\phi = \nabla f\cdot\phi.
}
Second,
\eq{
    e_j\cdot Q_{e_j} = \frac{1}{n-1}e_j\cdot\Big(e_j\cdot\nabla f\cdot\phi + e_j(f)\phi\Big) = -\nabla f\cdot\phi.
}
Third, since $\D^A\varphi=0$,
\eq{
    e_j\cdot R_{e_j}=\D^A\varphi - e_j\cdot P_{e_j} - e_j\cdot Q_{e_j} = 0.
}
Hence the claim follows.
\end{proof}

We next estimate the scalar components $\rRe\<iR_X,\phi\>$ in terms of $\abs{R}$. This is done by projecting onto the one-dimensional span of $\phi$ and the orthogonal complement of the Clifford-trace constraint from Lemma~\ref{lem1.4}(3).

\begin{lemma}\label{lem1.5}
We have
\eq{
    \sum_j \rRe\<iR_{e_j},\phi\>^2 \leq \frac{n-1}{n}\abs{R}^2.
}
\end{lemma}
\begin{proof}
Using Lemma~\ref{lem1.4}(3), we see that
\eq{
    \rRe\<iR_{e_j},e_j\cdot X\cdot\phi\>=0
}
for any vector field $X$. Hence
\eq{\label{eq:127}
    \rRe\<iR_X,\phi\> = \rRe\<iR_{e_j},\<X,e_j\>\phi+\frac{1}{n}e_j\cdot X\cdot\phi\>.
}
Note that
\eq{\label{eq:128}
    \sum_j \abs{\<X,e_j\>\phi+\frac{1}{n}e_j\cdot X\cdot\phi}^2 = \abs{X}^2 + \frac{1}{n}\abs{X}^2 + \frac{2}{n}\<X,e_j\>\rRe\<\phi,e_j\cdot X\cdot\phi\> = \frac{n-1}{n}\abs{X}^2.
}
Choosing $X=\rRe\<iR_{e_j},\phi\>e_j$ in \eqref{eq:127}, and using \eqref{eq:128} and Cauchy--Schwarz yield
\eq{\label{eq:128.1}
    \sum_j \rRe\<iR_{e_j},\phi\>^2 &= \rRe\<iR_X,\phi\> = \rRe\<iR_{e_j},\<X,e_j\>\phi+\frac{1}{n}e_j\cdot X\cdot\phi\> \\
    &\leq \sum_j\abs{R_{e_j}}\abs{\<X,e_j\>\phi+\frac{1}{n}e_j\cdot X\cdot\phi} \\
    &\leq \Big(\sum_j\abs{R_{e_j}}^2\Big)^{\frac{1}{2}}\Big(\sum_j\abs{\<X,e_j\>\phi+\frac{1}{n}e_j\cdot X\cdot\phi}^2\Big)^{\frac{1}{2}} \\
    &= \abs{R}\cdot\Big(\frac{n-1}{n}\sum_j \rRe\<iR_{e_j},\phi\>^2\Big)^{\frac{1}{2}}.
}
The desired estimate then follows.
\end{proof}

\subsection{The integral Schr\"odinger--Lichnerowicz formula}
Starting from \eqref{eq:S-L} and using the zero mode equation \eqref{eq:zero-mode}, we obtain an identity that will be the basis for our integral estimates:
\eq{\label{eq:4}
    0 = \rRe\<(\D^A)^2\varphi,\varphi\> = \rRe\<(\nabla^A)^*\nabla^A\varphi,\varphi\> + \frac{n(n-1)}{4}f^2 - \rRe\<i\,\rd A^\flat\cdot\varphi,\varphi\>.
}
Since
\eq{
    -\frac{1}{2}\Delta(f^2) = -\frac{1}{2}\Delta\abs{\varphi}^2 = \rRe\<(\nabla^A)^*\nabla^A\varphi,\varphi\> - \abs{\nabla^A\varphi}^2,
}
we have
\eq{\label{eq:5}
    \rRe\<(\nabla^A)^*\nabla^A\varphi,\varphi\> = -\frac{1}{2}\Delta(f^2) + \abs{\nabla^A\varphi}^2 = -f\Delta f - \abs{\nabla f}^2 + \abs{\nabla^A\varphi}^2.
}
Combining \eqref{eq:6}, \eqref{eq:4} and \eqref{eq:5} gives
\eq{\label{eq:7}
    f^2\rRe\<i\,\rd A^\flat\cdot\phi,\phi\> = \rRe\<i\rd A^\flat\cdot\varphi,\varphi\> = -f\Delta f + \frac{1}{n-1}\abs{\nabla f}^2 + \frac{n(n-1)}{4}f^2 + f^2\abs{S}^2.
}
Set $f=u^{\frac{n-1}{n-2}}$. Then \eqref{eq:7} becomes
\eq{\label{eq:8}
    u^2\rRe\<i\,\rd A^\flat\cdot\phi,\phi\> = \frac{n-1}{n-2}u\Big(-\Delta u+\frac{n(n-2)}{4}u\Big) + u^2\abs{S}^2.
}
Integrating yields
\eq{\label{eq:int}
   \int_{\S^n} u^2\rRe\<i\,\rd A^\flat\cdot\phi,\phi\> = \frac{n-1}{n-2}\int_{\S^n} \Big(\abs{\nabla u}^2+\frac{n(n-2)}{4}u^2\Big) + \int_{\S^n}  u^2\abs{S}^2.
}

\subsection{An associated \texorpdfstring{$2$}{2}-form}

In our previous work \cite{WZ26curl}, the three-dimensional case was treated using the $1$-form
\eq{\label{previous_one_form}
    \langle e_i\cdot \phi, \phi \rangle e^i.
}
In higher dimensions it is more effective to work instead with a naturally associated $2$-form built from the unit spinor $\phi$.
\begin{definition}\label{def3.2}
Define a $2$-form $\beta$ by
\eq{\label{eq:133}
   \beta = \sum_{j<k}\beta(e_j,e_k) e^j\wedge e^k\coloneqq \sum_{j<k}\rRe\<ie_j\cdot e_k\cdot\phi,\phi\> e^j\wedge e^k= -\sum_{j<k}\rIm\<e_j\cdot e_k\cdot\phi,\phi\>e^j\wedge e^k.
}
\end{definition}

\begin{lemma}\label{lem1.6}
We have
\eq{\label{eq:130}
    -\beta(e_j,e_k)\rIm\<\nabla^A_{e_j}\phi,\nabla^A_{e_k}\phi\> + \sum_j \rIm\<\nabla^A_{e_j}\phi,\phi\>^2 \leq \abs{\nabla^A\phi}^2.
}
\end{lemma}
\begin{proof}
Since $\abs{\phi}=1$, we have $\rRe\<\nabla^A_{e_j}\phi,\phi\>=0$, hence $\nabla^A\phi\perp\phi$. Moreover $i\phi\perp\phi$. We therefore decompose $\nabla^A\phi$ into its $i\phi$-component and the orthogonal remainder:
\eq{
    \nabla^A_{e_j}\phi = \Big(\nabla^A_{e_j}\phi - \rIm\<\nabla^A_{e_j}\phi,\phi\>i\phi\Big) + \rIm\<\nabla^A_{e_j}\phi,\phi\>i\phi.
}
In fact, one can check that  the two parts are orthogonal, since
\eq{
    \rRe\<\nabla^A_{e_j}\phi , \,\rIm\<\nabla^A_{e_j}\phi,\phi\>i\phi\> = \rIm\<\nabla^A_{e_j}\phi,\phi\> \rRe\<\nabla^A_{e_j}\phi,i\phi\> = \rIm\<\nabla^A_{e_j}\phi,\phi\>^2
}
and
\eq{
    \rRe\< \rIm\<\nabla^A_{e_j}\phi,\phi\>i\phi, \,\rIm\<\nabla^A_{e_j}\phi,\phi\>i\phi\> = \rIm\<\nabla^A_{e_j}\phi,\phi\>^2.
}
Therefore
\eq{\label{eq:129}
    \abs{\nabla^A\phi}^2 = \sum_j \abs{\nabla^A_{e_j}\phi - \rIm\<\nabla^A_{e_j}\phi,\phi\>i\phi}^2 + \sum_j \rIm\<\nabla^A_{e_j}\phi,\phi\>^2.
}
Since $\rRe\<\nabla^A_{e_j}\phi,\phi\>=0$, and hence $\rIm\<\nabla^A_{e_j}\phi,i\phi\>=0$, we have
\eq{\label{eq:131}
    \rIm\<\nabla^A_{e_j}\phi,\nabla^A_{e_k}\phi\> &= \rIm\<\nabla^A_{e_j}\phi - \rIm\<\nabla^A_{e_j}\phi,\phi\>i\phi,\nabla^A_{e_k}\phi - \rIm\<\nabla^A_{e_k}\phi,\,\phi\>i\phi\>.
}
Note that \eqref{eq:130} is independent of the choice of the basis. Choose a local basis $\{e_j\}$ such that $\beta$ is of the standard form
\eq{\label{eq:132}
    \beta = \sum_{m=1}^{[n/2]} \beta(e_{2m-1},e_{2m})e^{2m-1}\w e^{2m}.
}
By definition \eqref{eq:133}, we see that
\eq{\label{eq:134}
    \abs{\beta(e_j,e_k)} = \abs{\rRe\<ie_j\cdot e_k\cdot\phi,\phi\>} \leq 1.
}
Hence it follows by \eqref{eq:131}, \eqref{eq:132}, and \eqref{eq:134} that
\eq{\label{eq:135}
    &-\beta(e_j,e_k)\rIm\<\nabla^A_{e_j}\phi,\nabla^A_{e_k}\phi\> = -2\sum_{j<k} \beta(e_j,e_k)\rIm\<\nabla^A_{e_j}\phi,\nabla^A_{e_k}\phi\> \\
    &=-2\sum_{m=1}^{[n/2]} \beta(e_{2m-1},e_{2m})\rIm\<\nabla^A_{e_{2m-1}}\phi - \rIm\<\nabla^A_{e_{2m-1}}\phi,\phi\>i\phi,\nabla^A_{e_{2m}}\phi - \rIm\<\nabla^A_{e_{2m}}\phi,\,\phi\>i\phi\> \\
    &\leq\sum_{j=1}^{2[n/2]} \abs{\nabla^A_{e_j}\phi - \rIm\<\nabla^A_{e_j}\phi,\phi\>i\phi}^2 \leq \sum_{j=1}^n \abs{\nabla^A_{e_j}\phi - \rIm\<\nabla^A_{e_j}\phi,\phi\>i\phi}^2,
}
where we used Cauchy--Schwarz. Combining \eqref{eq:129} and \eqref{eq:135} gives the desired estimate \eqref{eq:130}.
\end{proof}

\begin{remark}
From \eqref{eq:132} and \eqref{eq:134} we immediately obtain the pointwise bound
\eq{\label{eq:l_beta} 
    \abs{\beta}^2 \leq \Big[\frac{n}{2}\Big]=\nu.
}
\end{remark}
This is where the constant $\nu$ in Theorem \ref{thm_main} comes from.

\begin{lemma}\label{lem1.7}
Recall $u=f^{\frac{n-2}{n-1}}$. Then
\eq{
    \rd^*(u^2\beta) = -2u^2f^{-1}\rIm\<R_{e_j},\phi\>e^j =-2u^2\rIm\<S_{e_j},\phi\>e^j.
}
\end{lemma}
\begin{proof}
Rewrite \eqref{eq:133} as
\eq{
    u^2\beta = -\sum_{j<k}\rIm\<e_j\cdot e_k\cdot u\phi,u\phi\>e^j\w e^k = -\frac{1}{2}\sum_{j,k}\rIm\<e_j\cdot e_k\cdot u\phi,u\phi\>e^j\w e^k.
}
Hence
\eq{\label{eq:138.1}
    \rd^*(u^2\beta) = -e_l\ip\nabla_{e_l}(u^2\beta) = -2\rIm\<\nabla^A_{e_j}(u\phi),u\phi\>e^j - 2\rIm\<e_j\cdot\D^A(u\phi),u\phi\>e^j.
}
Since $u=f^{\frac{n-2}{n-1}}$, we have
\eq{
    \nabla u = \frac{n-2}{n-1}u\nabla(\log f).
}
By \eqref{eq:138} we have
\eq{\label{eq:139}
    \rIm\<\nabla^A_{e_j}(u\phi),u\phi\> &= \rIm\<u\nabla^A_{e_j}\phi,u\phi\> = u\,\rIm\<T_{e_j}+S_{e_j},u\phi\> \\
    &= \frac{u^2}{n-1}\rIm\<e_j\cdot\nabla(\log f)\cdot\phi,\phi\> + u^2\rIm\<S_{e_j},\phi\>.
}
Using \eqref{eq:1} we have
\eq{
    \D^A(u\phi) = -\frac{u}{n-1}\nabla(\log f)\cdot\phi.
}
Hence
\eq{\label{eq:140}
    \rIm\<e_j\cdot\D^A(u\phi),u\phi\> = -\frac{u^2}{n-1}\rIm\<e_j\cdot\nabla(\log f)\cdot\phi,\phi\>.
}
The claim follows from \eqref{eq:138.1}, \eqref{eq:139}, and \eqref{eq:140}.
\end{proof}

\begin{lemma}\label{lem1.8}
We have
\eq{\label{eq:144}
    \rd A^\flat(e_j,e_k) = -e_j(\rIm\<\nabla^A_{e_k}\phi,\phi\>) + e_k(\rIm\<\nabla^A_{e_j}\phi,\phi\>) + \frac{1}{2}\beta(e_j,e_k) - 2\rIm\<\nabla^A_{e_j}\phi,\nabla^A_{e_k}\phi\>.
}
\end{lemma}
\begin{proof}
We begin with
\eq{\label{eq:142}
    &-e_j(\rIm\<\nabla^A_{e_k}\phi,\phi\>) + e_k(\rIm\<\nabla^A_{e_j}\phi,\phi\>) \\
    &= -\rIm\<\nabla^A_{e_j}\nabla^A_{e_k}\phi,\,\phi\> - \rIm\<\nabla^A_{e_k}\phi,\nabla^A_{e_j}\phi\> + \rIm\<\nabla^A_{e_k}\nabla^A_{e_j}\phi,\,\phi\> + \rIm\<\nabla^A_{e_j}\phi,\nabla^A_{e_k}\phi\> \\
    &= \rIm\<[\nabla^A_{e_k},\nabla^A_{e_j}]\phi,\,\phi\> + 2\rIm\<\nabla^A_{e_j}\phi,\nabla^A_{e_k}\phi\>.
}
Combining \eqref{eq:lem2.1} and \eqref{eq:142} we obtain
\eq{
    -e_j(\rIm\<\nabla^A_{e_k}\phi,\phi\>) + e_k(\rIm\<\nabla^A_{e_j}\phi,\phi\>) = -\frac{1}{2}\beta(e_j,e_k) + \rd A^\flat(e_j,e_k) + 2\rIm\<\nabla^A_{e_j}\phi,\nabla^A_{e_k}\phi\>.
}
The claim follows.
\end{proof}

\section{A first estimate using the critical Sobolev inequality}\label{sec4}

In this section we derive a first quantitative lower bound for $\norm{\rd A^\flat}_{L^{n/2}}$, improving the Frank--Loss estimate \cite{FL1}. The starting point is the integral identity \eqref{eq:int}, combined with the sharp critical Sobolev inequality on $\S^n$. We then explain how a stability refinement due to Dolbeault--Esteban--Figalli--Frank--Loss \cite{DEFFL22} yields the sharp bound \eqref{eq:main_ineq} in sufficiently large dimensions.

We start by estimating the left-hand side of \eqref{eq:int}. By Cauchy--Schwarz and \eqref{eq:l_beta}, we have
\eq{\label{eq:9}
    u^2 \rRe\<i\,\rd A^\flat\cdot\phi,\phi\> = u^2\<\rd A^\flat,\beta\> \leq \Big[\frac{n}{2}\Big]^{\frac{1}{2}}f^{\frac{2(n-2)}{n-1}}\abs{\rd A^\flat}.
}
Recall  the sharp critical Sobolev inequality on $\S^n$ 
\eq{\label{eq:10}
    \int_{\S^n} \Big(-u\Delta u + \frac{n(n-2)}{4}u^2\Big) \geq \frac{n(n-2)}{4}\,\omega_n^{\frac{2}{n}} \Big( \int_{\S^n} u^{\frac{2n}{n-2}} \Big)^{\frac{n-2}{n}}.
}
Inserting \eqref{eq:9} and \eqref{eq:10} into \eqref{eq:int} yields
\eq{\label{eq:11}
    \int_{\S^n} f^{\frac{2(n-2)}{n-1}}\abs{\rd A^\flat} \geq \Big[\frac{n}{2}\Big]^{-\frac{1}{2}} \frac{n(n-1)}{4}\,\omega_n^{\frac{2}{n}} \norm{f}_{\frac{2n}{n-1}}^{\frac{2(n-2)}{n-1}} + \Big[\frac{n}{2}\Big]^{-\frac{1}{2}}\int_{\S^n} f^{\frac{2(n-2)}{n-1}}\abs{S}^2.
}
Next, by H\"older's inequality,
\eq{\label{eq:12}
    \int_{\S^n} f^{\frac{2(n-2)}{n-1}}\abs{\rd A^\flat} \leq \norm{f}_{\frac{2n}{n-1}}^{\frac{2(n-2)}{n-1}}\cdot\norm{\rd A^\flat}_{\frac{n}{2}}.
}
Combining \eqref{eq:11} and \eqref{eq:12} we obtain
\eq{\label{eq:13}
\norm{\rd A^\flat}_{\frac{n}{2}} \ge \Big[\frac{n}{2}\Big]^{-\frac{1}{2}} \frac{n(n-1)}{4} \omega_n^{\frac{2}{n}} + \Big[\frac{n}{2}\Big]^{-\frac{1}{2}}\norm{f}_{\frac{2n}{n-1}}^{-\frac{2(n-2)}{n-1}} \int_{\S^n} f^{\frac{2(n-2)}{n-1}}\abs{S}^2.
}
If one discards the second term, 
then \eqref{eq:13} reduces to the Frank--Loss bound \cite{FL1}, which is exactly one half of the expected sharp constant. The key point is therefore to retain quantitative information from the second term.

\begin{lemma}
We have
\eq{\label{eq:155}
    \int u^2\<\rd A^\flat,\beta\> \leq  \frac 12 \Big[  \frac n2 \Big] \int u^2 + \frac{1}{n-1}\int u^2\abs{\nabla(\log f)}^2 + \frac{2n-1}{n}\int u^2\abs{S}^2.
}
\end{lemma}
\begin{proof}
We start from the pointwise identity \eqref{eq:144} in Lemma~\ref{lem1.8}. Pairing with $\beta$, multiplying by $u^2$, and integrating over $\S^n$ gives
\eq{\label{eq:147}
    \int u^2\<\rd A^\flat,\beta\> &= \int u^2 \sum_{j<k} \beta(e_j,e_k)\Big(-e_j(\rIm\<\nabla^A_{e_k}\phi,\phi\>) + e_k(\rIm\<\nabla^A_{e_j}\phi,\phi\>)\Big) \\
    &\quad + \frac{1}{2}\int u^2\abs{\beta}^2 - 2\int u^2 \sum_{j<k} \beta(e_j,e_k)\rIm\<\nabla^A_{e_j}\phi,\nabla^A_{e_k}\phi\>.
}
We estimate the first and the third integrals  on the right-hand side separately. For the first integral, integrating by parts and using Lemma~\ref{lem1.7} gives
\eq{\label{eq:145}
    &\int u^2 \sum_{j<k} \beta(e_j,e_k)\Big(-e_j(\rIm\<\nabla^A_{e_k}\phi,\phi\>) + e_k(\rIm\<\nabla^A_{e_j}\phi,\phi\>)\Big) \\
    &= \int u^2 \beta(e_j,e_k)\cdot e_k(\rIm\<\nabla^A_{e_j}\phi,\phi\>) = -\int (\nabla_{e_k}(u^2\beta))(e_j,e_k)\cdot \rIm\<\nabla^A_{e_j}\phi,\phi\> \\
    &= -\int \rd^*(u^2\beta)(e_j)\cdot \rIm\<\nabla^A_{e_j}\phi,\phi\> = 2\int u^2f^{-1}\rIm\<R_{e_j},\phi\>\cdot \rIm\<\nabla^A_{e_j}\phi,\phi\> \\
    &\leq \int u^2f^{-2}\sum_j \rIm\<R_{e_j},\phi\>^2 + \int u^2 \sum_j \rIm\<\nabla^A_{e_j}\phi,\phi\>^2.
}
For the third one, Lemma~\ref{lem1.6} yields
\eq{\label{eq:146}
    - 2\int u^2 \sum_{j<k} \beta(e_j,e_k)\rIm\<\nabla^A_{e_j}\phi,\nabla^A_{e_k}\phi\> \leq \int u^2\abs{\nabla^A\phi}^2 - \int u^2\sum_j\rIm\<\nabla^A_{e_j}\phi,\phi\>^2.
}
Substituting \eqref{eq:145} and \eqref{eq:146} into \eqref{eq:147} and using Lemma~\ref{lem1.5} gives
\eq{\label{eq:148}
    \int u^2\<\rd A^\flat,\beta\> \leq \frac{n-1}{n}\int u^2\abs{S}^2 + \frac{1}{2} \int u^2\abs{\beta}^2 + \int u^2\abs{\nabla^A\phi}^2.
}
Finally, we use the identities
\eq{\label{eq:149}
    \abs{\nabla^A\phi}^2 = \abs{T}^2 + \abs{S}^2 = \frac{1}{n-1}\abs{\nabla(\log f)}^2 + \abs{S}^2,
}
and the pointwise bound
\eq{\label{eq:150}
    \abs{\beta}^2 \leq \Big[\frac{n}{2}\Big],
}
which follows from \eqref{eq:l_beta}. Inserting \eqref{eq:149} and \eqref{eq:150} into \eqref{eq:148} gives \eqref{eq:155}.
\end{proof}

We now restate Proposition~\ref{prop_key0}.

\begin{proposition}\label{prop_key}
(1) If $n$ is odd, then
\eq{\label{eq:key_odd}
    \Big[\frac{n}{2}\Big]^{\frac{1}{2}} \norm{\rd A^\flat}_{\frac{n}{2}} \norm{u}_{\frac{2n}{n-2}}^2\geq \frac{2(n^2-3n+1)}{(n-2)^2}\int \abs{\nabla u}^2+ \frac{n(n-1)}{2}\int u^2. 
}
(2) If $n$ is even, then
\eq{\label{eq:key_even}
    \Big[\frac{n}{2}\Big]^{\frac{1}{2}} \norm{\rd A^\flat}_{\frac{n}{2}} \norm{u}_{\frac{2n}{n-2}}^2\geq \frac{2(n^2-3n+1)}{(n-2)^2}\int \abs{\nabla u}^2 + \Big(\frac{n(n-1)}{2}-\frac{n}{4(n-1)}\Big)\int u^2.
}
\end{proposition}
\begin{proof}
By definition \eqref{eq:133},\eq{   \rRe\<i\rd A^\flat\cdot\phi,\phi\> = \<\rd A^\flat,\beta\>.}
Therefore the integral identity \eqref{eq:int} can be rewritten as

\eq{\label{eq:158}
    \int u^2 \<\rd A^\flat,\beta\> = \frac{n-1}{n-2}\int\abs{\nabla u}^2 + \frac{n(n-1)}{4}\int u^2 + \int u^2\abs{S}^2.
}
Combining \eqref{eq:158} with \eqref{eq:155} and using
\eq{
    u^2\abs{\nabla(\log f)}^2 = \frac{(n-1)^2}{(n-2)^2}\abs{\nabla u}^2,
}
we obtain a lower bound for the remainder term:
\eq{
    \int u^2\abs{S}^2 \geq \frac{n(n-3)}{(n-2)^2} \int \abs{\nabla u}^2 + \Big( \frac{n^2}{4}-\frac{n}{2(n-1)}\Big[\frac{n}{2}\Big] \Big)\int u^2.
}
Substituting this estimate back into \eqref{eq:158} yields:

(1) if $n$ is odd, then
\eq{\label{eq:159}
    \int u^2 \<\rd A^\flat,\beta\> \geq \frac{2(n^2-3n+1)}{(n-2)^2}\int \abs{\nabla u}^2 + \frac{n(n-1)}{2}\int u^2;
}

(2) if $n$ is even, then
\eq{\label{eq:159.1}
    \int u^2 \<\rd A^\flat,\beta\> \geq \frac{2(n^2-3n+1)}{(n-2)^2}\int \abs{\nabla u}^2 + \Big(\frac{n(n-1)}{2}-\frac{n}{4(n-1)}\Big)\int u^2.
}
On the other hand, by Cauchy--Schwarz and H\"older (cf. \eqref{eq:9}),
\eq{\label{eq:160}
    \int u^2 \<\rd A^\flat,\beta\> \leq \Big[\frac{n}{2}\Big]^{\frac{1}{2}} \int u^2\abs{\rd A^\flat} \leq \Big[\frac{n}{2}\Big]^{\frac{1}{2}} \norm{\rd A^\flat}_{\frac{n}{2}} \norm{u}_{\frac{2n}{n-2}}^2.
}
The claim follows from \eqref{eq:159}, \eqref{eq:159.1}, and \eqref{eq:160}.
\end{proof}

At this point, if one directly uses the critical Sobolev inequality, then it leads to the following estimate.

\begin{corollary}
(1) If $n$ is odd, then
\eq{\label{eq:1st_estimate}
    \Big[\frac{n}{2}\Big]^{\frac{1}{2}}\norm{\rd A^\flat}_{\frac{n}{2}} \geq \Big( 1 - \frac{1}{(n-1)(n-2)} \Big) \frac{n(n-1)}{2}\omega_n^{\frac{2}{n}} + \frac{n}{2(n-2)}\norm{u}_{\frac{2n}{n-2}}^{-2} \int u^2.
}
(2) If $n$ is even, then
\eq{\label{eq:1st_estimate_even}
    \Big[\frac{n}{2}\Big]^{\frac{1}{2}}\norm{\rd A^\flat}_{\frac{n}{2}} \geq \Big( 1 - \frac{1}{(n-1)(n-2)} \Big) \frac{n(n-1)}{2}\omega_n^{\frac{2}{n}} + \Big(\frac{n}{2(n-2)}-\frac{n}{4(n-1)}\Big)\norm{u}_{\frac{2n}{n-2}}^{-2} \int u^2.
}
\end{corollary}
\begin{proof}
Using the critical Sobolev inequality we have
\eq{\label{eq:161}
    &\frac{2(n^2-3n+1)}{(n-2)^2}\int \abs{\nabla u}^2 + \frac{n(n-1)}{2}\int u^2 \\
    &= \Big( \frac{2(n-1)}{n-2} - \frac{2}{(n-2)^2} \Big) \Big( \int \abs{\nabla u}^2 + \frac{n(n-2)}{4}\int u^2 \Big) + \frac{n}{2(n-2)} \int u^2 \\
    &\geq \Big( \frac{2(n-1)}{n-2} - \frac{2}{(n-2)^2} \Big)\cdot \frac{n(n-2)}{4}\omega_n^{\frac{2}{n}}\norm{u}_{\frac{2n}{n-2}}^2 + \frac{n}{2(n-2)} \int u^2 \\
    &= \Big( 1 - \frac{1}{(n-1)(n-2)} \Big) \frac{n(n-1)}{2}\omega_n^{\frac{2}{n}}\norm{u}_{\frac{2n}{n-2}}^2 + \frac{n}{2(n-2)} \int u^2.
}
The claim follows from \eqref{eq:key_odd}, \eqref{eq:key_even}, and \eqref{eq:161}.
\end{proof}

The first term in the right-hand side of \eqref{eq:1st_estimate} for the odd case (and of \eqref{eq:1st_estimate_even}
for the even case) gives better bounds that are very close to the conjectured sharp inequality \eqref{eq:main_ineq}, but they are still not optimal. Together with the second term, we obtain the optimal  estimate \eqref{eq:main_ineq}  for  sufficiently large $n$.

In order to prove it we need the following important quantitative stability result for the sharp Sobolev inequality  established in \cite[Theorem 1.1]{DEFFL22} 
\eq{
    \int \Big( \abs{\nabla u}^2 + \frac{n(n-2)}{4}u^2 \Big) - \frac{n(n-2)}{4}\omega_n^{\frac{2}{n}}\norm{u}_{\frac{2n}{n-2}}^2 \geq \frac{c_0}{n}\inf_{v\in\mathcal{M}}\int \Big( \abs{\nabla (u-v)}^2 + \frac{n(n-2)}{4}(u-v)^2 \Big),
}
for a  $c_0>0$ which is independent of $n$ and $\mathcal{M}$ denotes the manifold of optimizers for the critical Sobolev inequality. Composing $u$ with a conformal transformation of $\S^n$, we may assume that the closest optimizer to $u$ is  a constant $v\equiv c$. In particular,
\eq{
    \inf_{v\in\mathcal{M}}\int \Big( \abs{\nabla (u-v)}^2 + \frac{n(n-2)}{4}(u-v)^2 \Big)
    = \int \Big( \abs{\nabla (u-c)}^2 + \frac{n(n-2)}{4}(u-c)^2 \Big)
    \geq \int \abs{\nabla u}^2.
}
Hence
\eq{\label{eq:162}
    \int \Big( \abs{\nabla u}^2 + \frac{n(n-2)}{4}u^2 \Big) - \frac{n(n-2)}{4}\omega_n^{\frac{2}{n}}\norm{u}_{\frac{2n}{n-2}}^2 \geq \frac{c_0}{n}\int \abs{\nabla u}^2.
}
Replacing the inequality in \eqref{eq:161} by the improved one \eqref{eq:162} gives
\eq{
    &\frac{2(n^2-3n+1)}{(n-2)^2}\int \abs{\nabla u}^2 + \frac{n(n-1)}{2}\int u^2 \\
    &\geq \Big( \frac{2(n-1)}{n-2} - \frac{2}{(n-2)^2} \Big)\cdot\Big( \frac{n(n-2)}{4}\omega_n^{\frac{2}{n}}\norm{u}_{\frac{2n}{n-2}}^2 + \frac{c_0}{n}\int \abs{\nabla u}^2 \Big) + \frac{n}{2(n-2)} \int u^2.
}
As long as
\eq{\label{eq:163}
    \Big( \frac{2(n-1)}{n-2} - \frac{2}{(n-2)^2} \Big)\cdot\frac{c_0}{n} \geq \frac{2}{(n-2)^2}, \quad\hbox{i.e.}\quad \frac{n}{n^2-3n+1}\le c_0,
}
we distinguish the following two cases.

(1) If $n$ is odd, then
\eq{
    &\frac{2(n^2-3n+1)}{(n-2)^2}\int \abs{\nabla u}^2 + \frac{n(n-1)}{2}\int u^2 \\
    &\geq \Big( \frac{2(n-1)}{n-2} - \frac{2}{(n-2)^2} \Big)\cdot \frac{n(n-2)}{4}\omega_n^{\frac{2}{n}}\norm{u}_{\frac{2n}{n-2}}^2 + \frac{n}{2(n-2)}\Big( \frac{4}{n(n-2)}\int \abs{\nabla u}^2 + \int u^2 \Big) \\
    &\geq \Big( \frac{2(n-1)}{n-2} - \frac{2}{(n-2)^2} \Big)\cdot \frac{n(n-2)}{4}\omega_n^{\frac{2}{n}}\norm{u}_{\frac{2n}{n-2}}^2 + \frac{n}{2(n-2)}\cdot \omega_n^{\frac{2}{n}}\norm{u}_{\frac{2n}{n-2}}^2 \\
    &= \frac{n(n-1)}{2}\omega_n^{\frac{2}{n}}\norm{u}_{\frac{2n}{n-2}}^2.
}
Combining with \eqref{eq:159} and \eqref{eq:160} gives \eqref{eq:main_ineq}.

(2) If $n$ is even, the same argument applies after replacing \eqref{eq:160} by the refined bound \eqref{eq:200} proved in Section~\ref{sec5}, and again yields \eqref{eq:main_ineq}.

Since \eqref{eq:163} holds for all sufficiently large $n$, this shows that the sharp lower bound \eqref{eq:main_ineq} follows immediately in high dimensions from the stability estimate of \cite{DEFFL22}.

\section{Proof of Theorem~\ref{thm_main}}\label{sec5}

We now prove Theorem~\ref{thm_main}. With Proposition~\ref{prop_key} in hand, instead of directly using the critical Sobolev inequality, we use an improved Aubin--Sobolev inequality (Theorem~\ref{thm_a_n0}), so that we are able to deal with all dimensions.

Note that given any closed $2$-form $B\in L^{\frac{n}{2}}$, there always exists a vector field $A\in L^n$, such that $B=\rd A^\flat$, see \cite{FL1}. Therefore we may assume $A\in L^n$. The regularity theory in \cites{FL1,WZ25b} then yields $\varphi\in L^{\frac{2n}{n-2}}$, hence $A\cdot\varphi\in L^2$, and therefore $\varphi\in W^{1,2}$. For the moment, we assume that $\varphi$ is nowhere vanishing; the extension across the possible zero set is justified by a  regularization argument in Subsection~\ref{sec6.5}.

By conformal invariance of the zero mode equation, we may use the Hersch balancing argument   to assume  that the function $u=f^{\frac{n-2}{n-1}}$ studied in the previous section satisfies the barycenter condition
\eq{\label{eq:barycenter}
    \int_{\S^n} xu^{\frac{2n}{n-2}} = 0.
}
Hence we can apply Theorem \ref{thm_a_n0} to this $u$.

It is clear that the statement of Theorem \ref{thm_a_n0} implies
\eq{\label{eq:goal_u}
    \frac{2(n^2-3n+1)}{(n-2)^2}\int \abs{\nabla u}^2 \geq \frac{n(n-1)}{2}\Big(\omega_n^{\frac{2}{n}}\norm{u}_{\frac{2n}{n-2}}^2 - \int u^2\Big).
}

\subsection{The odd \texorpdfstring{$n\geq5$}{n>=5} case}
For odd $n\geq5$, the conclusion follows immediately. Indeed, by \eqref{eq:key_odd} we have
\eq{\label{eq:odd_step1}
    \Big[\frac{n}{2}\Big]^{\frac{1}{2}} \norm{\rd A^\flat}_{\frac{n}{2}} \,\norm{u}_{\frac{2n}{n-2}}^2
    \geq \frac{2(n^2-3n+1)}{(n-2)^2}\int \abs{\nabla u}^2+ \frac{n(n-1)}{2}\int u^2.
}
Using \eqref{eq:goal_u} 
we have
\eq{\label{eq:}   
\begin{array}{rcl} \displaystyle \vspace{2mm}\Big[\frac{n}{2}\Big]^{\frac{1}{2}} \norm{\rd A^\flat}_{\frac{n}{2}} \,\norm{u}_{\frac{2n}{n-2}}^2  
    &\geq &
    \displaystyle\frac{n(n-1)}{2}\, \Big(\omega_n^{\frac{2}{n}}\norm{u}_{\frac{2n}{n-2}}^2-\int u^2\Big) +\frac {n(n-1) } 2 \int u^2 
    \\ 
&= & 
\displaystyle
\frac{n(n-1)}{2}\,\omega_n^{\frac{2}{n}}\norm{u}_{\frac{2n}{n-2}}^2,
\end{array}
}
and we finish the proof for this case.

\subsection{The even \texorpdfstring{$n\geq6$}{n≥6} case}\label{sec7.2}

In this subsection we treat the even-dimensional cases $n\geq6$.  As mentioned in the introduction, the previous proof is not enough for the even case, since we have only a weaker inequality \eqref{eq:key_even0}. However, 
when $n$ is even, the spinor bundle is reducible. We will use this property to improve the first  inequality  in \eqref{eq:key_even0}. This will be given in Proposition \ref{thm7.1} below. 

Since the spinor bundle is reducible, we have the chirality splitting of the spinor bundle:
\eq{
    \Sigma = \Sigma^+ \oplus \Sigma^-,
}
where
\eq{
    \Sigma^\pm = \{\psi\mid\omega\cdot\psi=\pm\psi\},
}
and $\omega=i^\nu e_1\cdots e_{2\nu}$ is the complex volume form. A spinor in $\Sigma^\pm$ is called a positive/negative spinor.

The Levi--Civita connection preserves chirality, while Clifford multiplication by a  vector exchanges it. Consequently, if $\phi$ is positive/negative, then
\eq{
    \phi, \quad i\phi, \quad \nabla^A_{e_j}\phi, \quad T_{e_j}, \quad S_{e_j}
}
all have the same chirality, whereas $\D$ exchanges $\Sigma^+$ and $\Sigma^-$.

At each point we decompose $\varphi$ into its chiral components,
\eq{
    \varphi = \varphi_+ + \varphi_- \quad\hbox{with}\quad \omega\cdot\varphi_\pm=\pm\varphi_\pm,
}
Then
\eq{
    \D\varphi = \D\varphi_+ + \D\varphi_-
}
and
\eq{
    iA\cdot\varphi = iA\cdot\varphi_+ + iA\cdot\varphi_-.
}
Since $\D$ and Clifford multiplication by $A$ both exchange chirality, the zero mode equation decouples into
\eq{
    \D\varphi_+=iA\cdot\varphi_+, \qquad \D\varphi_-=iA\cdot\varphi_-.
}
Thus, to estimate $\norm{\rd A^\flat}_{n/2}$ it suffices to treat one chirality, and we may assume $\varphi=\varphi_+$.

For simplicity, from now on we denote
\eq{
    F \coloneqq \rd A^\flat, \qquad \nu=\frac{n}{2}\geq3.
}
At each point $x\in\S^n$, the endomorphism $iF\cdot: \Sigma^+|_x\to\Sigma^+|_x$ is self-adjoint. We denote its largest eigenvalue by
\eq{
    \Lambda(x) \coloneqq \lambda_{\max}(iF\cdot)(x).
}
It suffices to prove the following estimate:
\begin{proposition}\label{thm7.1}
We have
\eq{\label{eq7:WTP}
    \norm{\Lambda}_\nu < \sqrt{\nu}\Big(1-\frac{1}{2(2\nu-1)^2}\Big)\norm{F}_\nu.
}
\end{proposition}

\begin{remark}
The constant in the right-hand side of \eqref{eq7:WTP} is not sharp. It is designed in order to cancel the deficit $-\frac{n}{4(n-1)}$ in \eqref{eq:159.1}, see the following discussion.
\end{remark}

\noindent{\it Proof of Theorem \ref{thm_main} for even $n\ge 6$}. Assuming \eqref{eq7:WTP}, we can bound the key quantity $\int u^2\<F,\beta\>$ as follows. Since
\eq{
    \<F,\beta\> = \rRe\<iF\cdot\phi,\phi\> \leq \Lambda,
}
we obtain
\eq{\label{eq:200}
    \int u^2\<F,\beta\> \leq \int u^2\Lambda \leq \norm{\Lambda}_\nu\,\norm{u}_{\frac{2\nu}{\nu-1}}^2
    < \sqrt{\frac{n}{2}}\Big(1-\frac{1}{2(n-1)^2}\Big)\norm{F}_{\frac{n}{2}}\,\norm{u}_{\frac{2n}{n-2}}^2.
}
On the other hand, since $n\geq6$ is even, combining \eqref{eq:159.1} with \eqref{eq:goal_u} yields
\eq{\label{eq:201}
    \int u^2 \<F,\beta\> &\geq \frac{2(n^2-3n+1)}{(n-2)^2}\int \abs{\nabla u}^2 + \Big(\frac{n(n-1)}{2}-\frac{n}{4(n-1)}\Big)\int u^2 \\
    &\geq \frac{n(n-1)}{2}\omega_n^{\frac 2{n}}\norm{u}_{\frac{2n}{n-2}}^2 - \frac{n}{4(n-1)}\int u^2 \\
    &\geq \frac{n(n-1)}{2}\Big(1-\frac{1}{2(n-1)^2}\Big)\omega_n^{\frac 2 {n}}\norm{u}_{\frac{2n}{n-2}}^2,
}
where the last step uses H\"older's inequality. Comparing \eqref{eq:200} and \eqref{eq:201} gives
\eq{
    \sqrt{\frac{n}{2}}\norm{F}_{\frac{n}{2}}\geq \frac{n(n-1)}{2}\omega_n^{\frac 2  {n} },
}
which is exactly \eqref{eq:main_ineq}. \qed

\smallskip

It remains to prove \eqref{eq7:WTP}, i.e. Proposition~\ref{thm7.1}.

\begin{proof}[Proof of Proposition~\ref{thm7.1}]
Fix a point and choose a local orthonormal basis $\{e_j\}$ so that the $2$-form $F$ is in block-diagonal normal form,
\eq{
    F = \sum_{j=1}^\nu \lambda_je^{2j-1}\w e^{2j}, \quad \lambda_j\in\R.
}
Then
\eq{\label{eq:202}
    \abs{F}^2 = \sum_{j=1}^\nu \lambda_j^2,\qquad F^\nu = \nu!\prod_{j=1}^\nu\lambda_j\cdot \rdV.
}
Without loss of generality we order the coefficients so that
\eq{
    \abs{\lambda_1}\geq\cdots\geq\abs{\lambda_\nu}.
}

For each $j$ define the commuting involutions
\eq{\label{eq:F_j}
    F_j\coloneqq ie^{2j-1}\w e^{2j}\cdot = ie_{2j-1}\cdot e_{2j}\cdot.
}
They satisfy
\eq{
    F_j^2 = \rid, \qquad F_j^*=F_j, \qquad [F_j,F_k]=0.
}
Thus, on $\Sigma^+$ the family $\{F_j\}_{j=1}^\nu$ can be simultaneously diagonalized. On a common eigenspinor the eigenvalues are $\epsilon_j\in\{\pm1\}$. Moreover, the chirality constraint imposes (see \cite{WZ25b} for the related pure spinor interpretation)
\eq{
    \prod_{j=1}^\nu \epsilon_j = \epsilon_0 \qquad\hbox{for some}\quad \epsilon_0\in\{\pm1\}.
}

Consequently, the eigenvalues of $iF\cdot$ are precisely
\eq{
    \sum_{j=1}^\nu \epsilon_j\lambda_j \qquad\hbox{with}\quad \prod_{j=1}^\nu \epsilon_j = \epsilon_0,
}
so
\eq{
    \Lambda = \max \Big\{ \sum_{j=1}^\nu\epsilon_j\lambda_j \,\Big|\, \epsilon_j=\pm1, \ \prod_{j=1}^\nu\epsilon_j=\epsilon_0 \Big\}.
}
We distinguish three cases according to the sign of $\epsilon_0\prod_{j=1}^\nu\lambda_j$.

\noindent\textbf{Case 1.} If $\epsilon_0\prod_{j=1}^\nu\lambda_j>0$, then one can choose $\epsilon_j={\rm sign}(\lambda_j)$. Hence
\eq{
    \Lambda = \sum_{j=1}^\nu \abs{\lambda_j}.
}

\noindent\textbf{Case 2.} If $\epsilon_0\prod_{j=1}^\nu\lambda_j=0$, then at least one $\lambda_{j_0}=0$, and one can choose all the other $\epsilon_j={\rm sign}(\lambda_j)$. Hence
\eq{
    \Lambda = \sum_{j=1}^\nu \abs{\lambda_j}.
}

\noindent\textbf{Case 3.} If $\epsilon_0\prod_{j=1}^\nu\lambda_j<0$, then there must be an odd number of $\lambda_j$'s such that $\epsilon_j=-{\rm sign}(\lambda_j)$. Hence
\eq{\label{eq:209}
    \Lambda = \sum_{j=1}^{\nu-1} \abs{\lambda_j} - \abs{\lambda_\nu}.
}

We claim that it suffices to prove the following pointwise estimate:
\eq{\label{eq:203}
    \Lambda^\nu \leq \nu^{\frac{\nu}{2}}\Big(1-\frac{1}{2(2\nu-1)^2}\Big)^\nu\abs{F}^\nu + \frac{\nu^\nu}{8} \epsilon_0\prod_{j=1}^\nu\lambda_j.
}
Indeed, by \eqref{eq:202},
\eq{
    \int \prod_{j=1}^\nu\lambda_j \,\rdV = \frac{1}{\nu!}\int F^\nu = \frac{1}{\nu!}\int \rd(A^\flat\w F^{\nu-1}) = 0.
}
Therefore, after integrating \eqref{eq:203} the last term drops out and we obtain \eqref{eq7:WTP}.

\

\noindent\textit{Proof of Claim \eqref{eq:203}}: We first prove it for

{\it Case 1}  and {\it Case 2.}  
In these cases we have
\eq{\label{eq:210}
    \epsilon_0\prod_{j=1}^\nu\lambda_j = \prod_{j=1}^\nu\abs{\lambda_j} \geq 0.
}
By the scaling-invariance, we may normalize
\eq{\label{eq:205}
    \abs{F}^2 = \sum_{j=1}^\nu\abs{\lambda_j}^2 = 1.
}
It follows that
\eq{\label{eq:203.1}
    \#\{j:\abs{\lambda_j}<\nu^{-\frac{1}{2}}\} \leq \nu-1.
}
If
\eq{\label{eq:204}
    \Lambda = \sum_{j=1}^\nu\abs{\lambda_j} \leq \sqrt{\nu}\Big(1-\frac{1}{2(2\nu-1)^2}\Big),
}
then \eqref{eq:203} is immediate. We may therefore assume the complementary case
\eq{
    \Lambda = \sum_{j=1}^\nu\abs{\lambda_j} > \sqrt{\nu}\Big(1-\frac{1}{2(2\nu-1)^2}\Big).
}
This implies
\eq{\label{eq:206}
    \sum_{j=1}^\nu \big(1-\sqrt{\nu}\,\abs{\lambda_j}\big)^2 = 2\nu -2\sqrt{\nu}\Lambda < \frac{\nu}{(2\nu-1)^2}.
}
Using Cauchy--Schwarz together with \eqref{eq:203.1} and \eqref{eq:206} we obtain
\eq{
    \sum_{\abs{\lambda_j}<\nu^{-1/2}}\big(1-\sqrt{\nu}\,\abs{\lambda_j}\big) &\leq \Big( (\nu-1)\sum_{\abs{\lambda_j}<\nu^{-1/2}}\big(1-\sqrt{\nu}\,\abs{\lambda_j}\big)^2 \Big)^{\frac{1}{2}} \\
    &\leq \Big( (\nu-1)\sum_{j=1}^\nu\big(1-\sqrt{\nu}\,\abs{\lambda_j}\big)^2 \Big)^{\frac{1}{2}} = \Big((\nu-1)\big(2\nu-2\sqrt{\nu}\Lambda\big)\Big)^{\frac{1}{2}} \\
    &< \frac{\sqrt{\nu(\nu-1)}}{2\nu-1} < \frac{1}{2}.
}
It follows that
\eq{\label{eq:207}
    \prod_{j=1}^\nu \sqrt{\nu}\,\abs{\lambda_j} &\geq \prod_{\abs{\lambda_j}<\nu^{-1/2}}\sqrt{\nu}\,\abs{\lambda_j} = \prod_{\abs{\lambda_j}<\nu^{-1/2}}\Big\{1-\big(1-\sqrt{\nu}\,\abs{\lambda_j}\big)\Big\} \\
    &\geq 1 - \sum_{\abs{\lambda_j}<\nu^{-1/2}}\big(1-\sqrt{\nu}\,\abs{\lambda_j}\big) > \frac{1}{2}.
}
The normalization \eqref{eq:205} implies that
\eq{\label{eq:208}
    \Lambda = \sum_{j=1}^\nu\abs{\lambda_j} \leq \sqrt{\nu}\Big(\sum_{j=1}\abs{\lambda_j}^2\Big)^{\frac{1}{2}} = \sqrt{\nu}.
}
Combining \eqref{eq:207} and \eqref{eq:208} yields
\eq{
    &\Lambda^\nu - \nu^{\frac{\nu}{2}}\Big(1-\frac{1}{2(2\nu-1)^2}\Big)^\nu
    \leq \nu^{\frac{\nu}{2}}\Big\{ 1 - \Big(1-\frac{1}{2(2\nu-1)^2}\Big)^\nu \Big\} \\
    &\leq \frac{\nu^{\frac{\nu}{2}}\nu}{2(2\nu-1)^2}
    < \frac{\nu^{\frac{\nu}{2}}}{16}
    < \frac{\nu^{\frac{\nu}{2}}}{8}\prod_{j=1}^\nu \sqrt{\nu}\,\abs{\lambda_j}
    = \frac{\nu^\nu}{8}\prod_{j=1}^\nu\abs{\lambda_j}
    = \frac{\nu^\nu}{8}\epsilon_0\prod_{j=1}^\nu\lambda_j,
}
where we used \eqref{eq:210} in the last step.

It remains to show 

{\it Case 3.} 
In this case,
\eq{\label{eq:211}
    \epsilon_0\prod_{j=1}^\nu\lambda_j = -\prod_{j=1}^\nu\abs{\lambda_j} \leq 0.
}
Moreover, by \eqref{eq:209} and Cauchy--Schwarz,

\eq{\label{eq:212}
    \Lambda \leq \Big((\nu-1)\sum_{j=1}^{\nu-1}\abs{\lambda_j}^2\Big)^{\frac{1}{2}} \leq (\nu-1)^{\frac{1}{2}}\abs{F}.
}
Note that
\eq{\label{eq:213}
    (\nu-1)^{\frac{\nu}{2}} = \Big(1-\frac{1}{\nu}\Big)^{\frac{\nu}{2}}\nu^{\frac{\nu}{2}} < e^{-\frac{1}{2}}\nu^{\frac{\nu}{2}} < \frac{5}{8}\nu^{\frac{\nu}{2}}.
}
Combining \eqref{eq:212} and \eqref{eq:213} we obtain
\eq{\label{eq:214}
    \Lambda^\nu \leq (\nu-1)^{\frac{\nu}{2}}\abs{F}^\nu < \frac{5}{8}\nu^{\frac{\nu}{2}}\abs{F}^\nu.
}
On the other hand, by AM--GM,
\eq{\label{eq:215}
    \prod_{j=1}^\nu\abs{\lambda_j} \leq \Big(\nu^{-1}\sum_{j=1}^\nu\abs{\lambda_j}^2\Big)^{\frac{\nu}{2}} = \nu^{-\frac{\nu}{2}}\abs{F}^\nu.
}
Finally,
\eq{\label{eq:216}
    \Big(1-\frac{1}{2(2\nu-1)^2}\Big)^\nu \geq 1 - \frac{\nu}{2(2\nu-1)^2} > \frac{3}{4}.
}
Putting \eqref{eq:211}, \eqref{eq:214}, \eqref{eq:215}, and \eqref{eq:216} together gives
\eq{
    \Lambda^\nu - \frac{\nu^\nu}{8}\epsilon_0\prod_{j=1}^\nu\lambda_j < \nu^{\frac{\nu}{2}}\Big(1-\frac{1}{2(2\nu-1)^2}\Big)^\nu\abs{F}^\nu,
}
which is exactly \eqref{eq:203} in Case~3. This completes the proof.
\end{proof}

\subsection{The equality case for \texorpdfstring{$n\geq5$}{n>=5}}

Since the inequality in \eqref{eq7:WTP} (and hence \eqref{eq:200}) is strict, the bound \eqref{eq:main_ineq} is also strict when $n$ is even. Thus equality can occur only in odd dimensions. In that case, equality holds if and only if, modulo conformal and gauge transformations, the following equalities occur in the chain of estimates:
\begin{enumerate}
    \item equality in \eqref{eq:160}, i.e. $\rd A^\flat = c(x)\beta$ with $c(x)>0$, and $\abs{\rd A^\flat}^{\frac{n}{2}}=Cu^{\frac{2n}{n-2}}$;
    \item equalities in \eqref{eq:134} and \eqref{eq:135}, i.e. $\beta = \sum_{j=1}^{[n/2]}\epsilon_j e^{2j-1}\w e^{2j}$ with $\epsilon_j\in\{\pm1\}$. In particular $\phi$ is a simultaneous eigenspinor of the involutions $F_j$ in \eqref{eq:F_j}. Consequently there exists a unique unit vector field $\xi$ such that $\xi\cdot\phi=i\phi$ and $\ker(\beta)={\rm span}_\R\{\xi\}$;
    \item equality in \eqref{eq:goal_u}, i.e. $u\equiv{\rm const}>0$, hence $\abs{\varphi}\equiv{\rm const}>0$;
    \item equality in the final inequality in \eqref{eq:135}, i.e. $\nabla^A_{\xi} \phi - \rIm\<\nabla^A_{\xi}\phi,\phi\>i\phi = 0$;
    \item equality in \eqref{eq:128.1}, i.e.
    $S_{e_j}=\frac{i}{n-1}\big(n\rIm\<S_{e_j},\phi\>\phi + e_j\cdot \rIm\<S_{e_k},\phi\>e_k\cdot\phi\big)$.
\end{enumerate}

We briefly explain how these conditions force the geometric characterization. From (2) we obtain $\abs{\beta}^2=\frac{n-1}{2}$. Conditions (1) and (3) imply $c\equiv n$, hence $\rd A^\flat=n\beta$. Condition (3) gives $T=0$, and therefore $\nabla^A\phi=S$. Next, (4) yields $S_\xi=\rIm\<S_\xi,\phi\>i\phi$. Combining this with (5) gives $\rIm\<S_{e_j},\phi\>e_j=\frac{n-1}{2}\epsilon\xi$ for some $\epsilon\in\{\pm1\}$. Substituting back into (5) yields
$S_{e_j}=-\frac{\epsilon}{2}e_j\cdot\phi+\frac{n\epsilon}{2}\<\xi,e_j\>i\phi$.
Differentiating $\xi\cdot\phi=i\phi$ and choosing $A=-\frac{n\epsilon}{2}\xi$ (modulo gauge transformations), we arrive at
$\nabla_{e_j}\phi=-\frac{\epsilon}{2}e_j\cdot\phi$.
Thus $\phi$ (and hence $\varphi$) is a Killing spinor, and $A$ is a real multiple of the Reeb field associated with $\varphi$.

The reversed direction is easy to check.

\subsection{The case \texorpdfstring{$n=4$}{n=4}}

We now prove the main result in the remaining borderline dimension $n=4$. Let $\phi$ be a negative spinor as in Section~\ref{appendix_Clifford}.

\begin{theorem}\label{thm_n=4}
For $n=4$ we have
\eq{\label{eq:4D}
    \norm{\rd A^\flat}_2 > 6\,\omega_4^{\frac{1}{2}}.
}
\end{theorem}

We record the specializations of the identities from Sections~3 and 4 to the four-dimensional situation. First, Lemma~\ref{lem1.2} gives
\eq{\label{eq51}
    \abs{\nabla^A\phi}^2 = \abs{T}^2 + \abs{S}^2 = \frac{1}{3}\abs{\nabla(\log f)}^2 + \abs{S}^2 = \frac{3}{4}u^{-2}\abs{\nabla u}^2 + \abs{S}^2;
}
next, Lemma~\ref{lem1.7} becomes
\eq{\label{eq52}
    \rd^*(u^2\beta) = -2u^2\rIm\<S_{e_j},\phi\>e^j \quad\hbox{for}\quad u=f^{\frac{2}{3}};
}
and Lemma~\ref{lem1.8} becomes
\eq{\label{each term}
    \rd A^\flat(e_j,e_k) = -e_j(\rIm\<\nabla^A_{e_k}\phi,\phi\>) + e_k(\rIm\<\nabla^A_{e_j}\phi,\phi\>) + \frac{1}{2}\beta(e_j,e_k) - 2\rIm\<\nabla^A_{e_j}\phi,\nabla^A_{e_k}\phi\>;
}
Moreover, the integral Schr\"odinger--Lichnerowicz identity \eqref{eq:158} reads
\eq{\label{eq54}
    \int u^2 \<\rd A^\flat,\beta\> = \frac{3}{2}\int\abs{\nabla u}^2 + 3\int u^2 + \int u^2\abs{S}^2;
}
and the integration-by-parts step \eqref{eq:145} becomes
\eq{\label{eq55}
     &\int u^2 \sum_{j<k} \beta(e_j,e_k)\Big(-e_j(\rIm\<\nabla^A_{e_k}\phi,\phi\>) + e_k(\rIm\<\nabla^A_{e_j}\phi,\phi\>)\Big) \\
     &= 2\int u^2\rIm\<S_{e_j},\phi\>\cdot \rIm\<\nabla^A_{e_j}\phi,\phi\> \\
     &= \int u^2\sum_j\rIm\<\nabla^A_{e_j}\phi,\phi\>^2 + \int u^2\sum_j\rIm\<S_{e_j},\phi\>^2 - \int u^2\sum_j\rIm\<T_{e_j},\phi\>^2,
}
where we used the decomposition $\nabla^A_{e_j}\phi=T_{e_j}+S_{e_j}$.

To exploit \eqref{eq54}--\eqref{eq55} we will use a few pointwise identities that depend on the four-dimensional Clifford algebra; see Appendix~\ref{appendix_Clifford} and the chirality discussion at the beginning of Subsection~\ref{sec7.2}.

\begin{lemma}\label{lem_beta}
Under suitable coordinates, we have
\eq{\label{eq55.00}
    \beta = e^1\w e^2 - e^3\w e^4,
}
hence
\eq{
    \abs{\beta}^2 = 2.
}
\end{lemma}
\begin{proof}
By definition,
\eq{
    \beta = -\sum_{1\leq j<k\leq 4}\rIm\<e_j\cdot e_k\cdot\phi,\phi\>e^j\w e^k.
}
In the explicit Clifford model from Appendix~\ref{appendix_Clifford}, the only nonzero coefficients are
\eq{
    \rIm\<e_1\cdot e_2\cdot \phi,\phi\> = -1, \qquad \rIm\<e_3\cdot e_4\cdot \phi,\phi\> = 1.
}
Substituting these values yields \eqref{eq55.00} and hence $\abs{\beta}^2=2$.
\end{proof}

\begin{lemma}\label{lem3.2}
We have
\eq{
    \sum_j\rIm\<T_{e_j},\phi\>^2 = \frac{1}{9}\abs{\nabla(\log f)}^2 = \frac{1}{4}u^{-2}\abs{\nabla u}^2.
}
\end{lemma}
\begin{proof}
Fix a point. If $\nabla(\log f)=0$, the claim is immediate; otherwise we choose an orthonormal basis $\{e_1,e_2,e_3,e_4\}$ with $e_1=\frac{\nabla(\log f)}{\abs{\nabla(\log f)}}$. Recalling the definition \eqref{eq:138},
\eq{
    T_{e_j} = \frac{1}{3}\Big(e_j\cdot\nabla(\log f)\cdot\phi + e_j(\log f)\phi\Big) = \frac{1}{3}\abs{\nabla(\log f)}\big(e_j\cdot e_1\cdot\phi + \delta_{j1}\phi\big).
}
Hence
\eq{\label{eq55.0}
    T_{e_1}=0, \quad T_{e_j}=\frac{1}{3}\abs{\nabla(\log f)}\,e_j\cdot e_1\cdot\phi, \quad j=2,3,4.
}
Using \eqref{eq:B3}, we see that
\eq{\label{eq55.1}
    T_{e_2}=\frac{1}{3}\abs{\nabla(\log f)}\,i\phi, \qquad T_{e_3}=\frac{1}{3}\abs{\nabla(\log f)}\,\phi_2, \qquad T_{e_4}=-\frac{1}{3}\abs{\nabla(\log f)}\,i\phi_2.
}
Hence
\eq{
    \sum_j\rIm\<T_{e_j},\phi\>^2 &= \frac{1}{9}\abs{\nabla(\log f)}^2 .
}
The conclusion follows.
\end{proof}

\begin{lemma}\label{lem3.3}
Set
\eq{
    \mathcal{A} \coloneqq \abs{\nabla^A\phi}^2 + \beta(e_j,e_k)\rIm\<\nabla^A_{e_j}\phi,\nabla^A_{e_k}\phi\> - \sum_j\rIm\<\nabla^A_{e_j}\phi,\phi\>^2.
}
Then we have $\mathcal{A} \geq 0$, and
\eq{
    \mathcal{A} = 2\abs{S}^2 - 3\sum_j\rIm\<S_{e_j},\phi\>^2.
}
\end{lemma}
\begin{proof}
 Appendix~\ref{appendix_Clifford} shows that $\{\phi_1=\phi,\phi_2\}$ spans one chiral subspace, and $\{\phi_3,\phi_4\}$ spans the other.

Note that $S_{e_j}\perp\phi$, and the $i\phi$-component is
\eq{\label{eq55.2}
    \rRe\<S_{e_j},i\phi\> = \rIm\<S_{e_j},\phi\> \eqcolon a_j \in\R.
}
Hence we may decompose
\eq{\label{eq55.2.1}
    S_{e_j} = a_j\,i\phi + z_j\phi_2, \quad z_j\in\mathbb{C}.
}
Then the Clifford trace-free condition (recall Lemma~\ref{lem1.4}(3)) implies 
\eq{
    0 &= e_j\cdot S_{e_j} = ia_je_j\cdot\phi + z_je_j\cdot\phi_2 \\
    &= i(-a_1\phi_4-a_2i\phi_4-a_3\phi_3+a_4i\phi_3)+(-z_1\phi_3+z_2i\phi_3+z_3\phi_4+z_4i\phi_4) \\
    &= (-ia_3-a_4-z_1+iz_2)\phi_3 + (-ia_1+a_2+z_3+iz_4)\phi_4,
}
where we used \eqref{eq:B1} and \eqref{eq:B2}. Hence
\eq{
    -ia_3-a_4-z_1+iz_2=0, \qquad -ia_1+a_2+z_3+iz_4=0,
}
or equivalently
\eq{\label{eq55.2.2}
    z_1-iz_2 = -i(a_3-ia_4), \qquad z_3+iz_4 = i(a_1+ia_2).
}

Note that
\eq{
    \sum_j \abs{ \nabla^A_{e_j}\phi - \rIm\<\nabla^A_{e_j}\phi,\phi\>i\phi }^2 = \abs{\nabla^A\phi}^2 - \sum_j\rIm\<\nabla^A_{e_j}\phi,\phi\>^2.
}
Since $\nabla^A_{e_j}\phi\perp\phi$, we have
\eq{
    \rIm\<\nabla^A_{e_1}\phi,\nabla^A_{e_2}\phi\> = \rIm\<\nabla^A_{e_1}\phi - \rIm\<\nabla^A_{e_1}\phi,\phi\>i\phi,\nabla^A_{e_2}\phi - \rIm\<\nabla^A_{e_2}\phi,\phi\>i\phi\>,
}
and
\eq{
    \rIm\<\nabla^A_{e_3}\phi,\nabla^A_{e_4}\phi\> = \rIm\<\nabla^A_{e_3}\phi - \rIm\<\nabla^A_{e_3}\phi,\phi\>i\phi,\nabla^A_{e_4}\phi - \rIm\<\nabla^A_{e_4}\phi,\phi\>i\phi\>.
}
Hence Lemma~\ref{lem_beta} implies that
\eq{\label{eq55.3}
    \mathcal{A} &= \abs{\nabla^A\phi}^2 + \beta(e_j,e_k)\rIm\<\nabla^A_{e_j}\phi,\nabla^A_{e_k}\phi\> - \sum_j\rIm\<\nabla^A_{e_j}\phi,\phi\>^2 \\
    &= \abs{\nabla^A\phi}^2 + 
    2\rIm\<\nabla^A_{e_1}\phi,\nabla^A_{e_2}\phi\> - 2\rIm\<\nabla^A_{e_3}\phi,\nabla^A_{e_4}\phi\> - \sum_j\rIm\<\nabla^A_{e_j}\phi,\phi\>^2 \\
    &= \sum_j \abs{ \nabla^A_{e_j}\phi - \rIm\<\nabla^A_{e_j}\phi,\phi\>i\phi }^2 + 2\rIm\<\nabla^A_{e_1}\phi - \rIm\<\nabla^A_{e_1}\phi,\phi\>i\phi,\nabla^A_{e_2}\phi - \rIm\<\nabla^A_{e_2}\phi,\phi\>i\phi\> \\
    &\quad - 2\rIm\<\nabla^A_{e_3}\phi - \rIm\<\nabla^A_{e_3}\phi,\phi\>i\phi,\nabla^A_{e_4}\phi - \rIm\<\nabla^A_{e_4}\phi,\phi\>i\phi\> \\
    &= \Abs{ \big(\nabla^A_{e_1}\phi - \rIm\<\nabla^A_{e_1}\phi,\phi\>i\phi\big) + i\big(\nabla^A_{e_2}\phi - \rIm\<\nabla^A_{e_2}\phi,\phi\>i\phi\big) }^2 \\
    &\quad + \Abs{ \big(\nabla^A_{e_3}\phi - \rIm\<\nabla^A_{e_3}\phi,\phi\>i\phi\big) - i\big(\nabla^A_{e_4}\phi - \rIm\<\nabla^A_{e_4}\phi,\phi\>i\phi\big) }^2 \geq 0.
}
Moreover, since
\eq{
    \nabla^A_{e_j}\phi - \rIm\<\nabla^A_{e_j}\phi,\phi\>i\phi = S_{e_j} + T_{e_j} - \rIm\<S_{e_j},\phi\>i\phi - \rIm\<T_{e_j},\phi\>i\phi,
}
using \eqref{eq55.0}, \eqref{eq55.1}, \eqref{eq55.2}, and \eqref{eq55.2.1} we have
\eq{\label{eq55.4}
    \nabla^A_{e_1}\phi - \rIm\<\nabla^A_{e_1}\phi,\phi\>i\phi &= S_{e_1} - a_1i\phi = z_1\phi_2 \\
    \nabla^A_{e_2}\phi - \rIm\<\nabla^A_{e_2}\phi,\phi\>i\phi &= S_{e_2} - a_2i\phi = z_2\phi_2 \\
    \nabla^A_{e_3}\phi - \rIm\<\nabla^A_{e_3}\phi,\phi\>i\phi &= S_{e_3} - a_3i\phi + \frac{1}{3}\abs{\nabla(\log f)}\phi_2 = z_3\phi_2 + \frac{1}{3}\abs{\nabla(\log f)}\phi_2 \\
    \nabla^A_{e_4}\phi - \rIm\<\nabla^A_{e_4}\phi,\phi\>i\phi &= S_{e_4} - a_4i\phi - \frac{1}{3}\abs{\nabla(\log f)}i\phi_2 = z_4\phi_2 - \frac{1}{3}\abs{\nabla(\log f)}i\phi_2.
}
Combining \eqref{eq55.3} and \eqref{eq55.4} gives
\eq{\label{eq55.5}
    \mathcal{A} = \abs{ z_1+iz_2 }^2 + \abs{ z_3-iz_4 }^2.
}
Since $a_j\in\R$, by \eqref{eq55.2.2} we have
\eq{\label{eq55.6}
    \abs{z_1-iz_2}^2 + \abs{z_3+iz_4}^2 = \abs{a_3-ia_4}^2 + \abs{a_1+ia_2}^2 = \sum_j a_j^2.
}
Finally, \eqref{eq55.2.1} implies that
\eq{\label{eq55.7}
    \abs{S}^2 = \sum_j\abs{S_{e_j}}^2 =  \sum_j\abs{a_j\,i\phi+z_j\phi_2}^2 = \sum_ja_j^2 + \sum_j\abs{z_j}^2.
}
Using the elementary identity
\eq{
    \abs{ z_1+iz_2 }^2 + \abs{ z_3-iz_4 }^2 + \abs{z_1-iz_2}^2 + \abs{z_3+iz_4}^2 = 2\sum_j\abs{z_j}^2,
}
we have by \eqref{eq55.5}, \eqref{eq55.6}, and \eqref{eq55.7} that
\eq{
    \mathcal{A} = 2\sum_j\abs{z_j}^2 - \sum_j a_j^2 = 2\abs{S}^2 - 3\sum_ja_j^2.
}
Hence we complete the proof.
\end{proof}

We now prove Theorem~\ref{thm_n=4}.

\begin{proof}[Proof of Theorem~\ref{thm_n=4}]
Using $\abs{\beta}^2=2$ from Lemma~\ref{lem_beta}, together with \eqref{eq51}, \eqref{eq:147}, \eqref{eq55}, Lemma~\ref{lem3.2}, and Lemma~\ref{lem3.3}, we obtain
\eq{\label{eq56}
    \int u^2\<\rd A^\flat,\beta\> &= \int u^2\sum_j\rIm\<\nabla^A_{e_j}\phi,\phi\>^2 + \int u^2\sum_j\rIm\<S_{e_j},\phi\>^2 - \int u^2\sum_j\rIm\<T_{e_j},\phi\>^2 \\
    &\quad + \int u^2 - \int u^2\beta(e_j,e_k)\rIm\<\nabla^A_{e_j}\phi,\nabla^A_{e_k}\phi\> \\
    &= \frac{2}{3}\int u^2\abs{S}^2 - \frac{1}{3}\int u^2\mathcal{A} - \frac{1}{4}\int \abs{\nabla u}^2 + \int u^2 - \int u^2\mathcal{A} + \int u^2\abs{\nabla^A\phi}^2 \\
    &= \frac{1}{2}\int \abs{\nabla u}^2 + \int u^2 - \frac{4}{3}\int u^2\mathcal{A} + \frac{5}{3}\int u^2\abs{S}^2.
}
Lemma~\ref{lem3.3} gives $\mathcal{A}\geq0$. Substituting \eqref{eq56} into \eqref{eq54} therefore yields
\eq{
    \frac{2}{3}\int u^2\abs{S}^2 = \int \abs{\nabla u}^2 + 2\int u^2 + \frac{4}{3}\int u^2\mathcal{A} \geq \int \abs{\nabla u}^2 + 2\int u^2.
}
Combining this with \eqref{eq54}, we obtain
\eq{\label{eq57}
    \int u^2 \<\rd A^\flat,\beta\> \geq 3\Big(\int\abs{\nabla u}^2 + 2\int u^2\Big)
    \geq 6\,\omega_4^{\frac{1}{2}}\norm{u}_4^2,
}
where the last step uses the sharp critical Sobolev inequality on $\S^4$.

Finally, we claim that
\eq{\label{eq58}
    \int u^2 \<\rd A^\flat,\beta\> \leq \norm{\rd A^\flat}_2\norm{u}_4^2,
}
which is strictly stronger than \eqref{eq:160}. In fact, under the standard orientation, \eqref{eq55.00} implies that
\eq{
    *\beta = -\beta,
}
which means $\beta$ is a $-1$ eigenform of the Hodge star. We decompose $\rd A^\flat$ by
\eq{
    \rd A^\flat = (\rd A^\flat)^+ + (\rd A^\flat)^-,
}
where $(\rd A^\flat)^\pm$ is the $\pm1$ eigen-component. It follows
\eq{
    0 = \int \rd A^\flat\w\rd A^\flat = \norm{(\rd A^\flat)^+}_2^2 - \norm{(\rd A^\flat)^-}_2^2.
}
Hence
\eq{
    \norm{(\rd A^\flat)^+}_2^2 = \norm{(\rd A^\flat)^-}_2^2 = \frac{1}{2}\norm{\rd A^\flat}_2^2.
}
Since $\abs{\beta}^2=2$, we may estimate
\eq{
    \int u^2 \<\rd A^\flat,\beta\> = \int u^2 \<(\rd A^\flat)^-,\beta\>
    \leq \sqrt{2}\int u^2\abs{(\rd A^\flat)^-}
    \leq \sqrt{2}\norm{(\rd A^\flat)^-}_2\norm{u}_4^2
    = \norm{\rd A^\flat}_2\norm{u}_4^2.
}
Together with \eqref{eq57}, this proves \eqref{eq:4D}.

Finally, the inequality is strict. Indeed, equality would force $u$ to be constant modulo conformal transformations, hence $u>0$ everywhere, and force $\phi$ to be nowhere vanishing and hence smooth. Then \eqref{eq55.00} would imply that $\beta$ defines a smooth almost complex structure on $\S^4$, which is impossible. Therefore \eqref{eq:4D} is strict.
\end{proof}

\subsection{Regularization}\label{sec6.5}

The arguments above were carried out under the simplifying assumption that $\varphi$ is nowhere vanishing. We now briefly justify how to pass to the general case when $\varphi$ may have zeros.

Since $f=\abs{\varphi}\in W^{1,2}$, we have $\nabla f=0$ almost everywhere on the zero set $\{f=0\}$. Hence all pointwise identities extend across $\{f=0\}$ in a trivial way. The remaining issue is to justify the integral Schr\"odinger--Lichnerowicz identity \eqref{eq:158}. For this we introduce the regularization used in \cite{Frank_Loss_2024}
\eq{
    f_\epsilon\coloneqq \sqrt{\abs{\varphi}^2+\epsilon^2}=\sqrt{f^2+\epsilon^2},\quad u_\epsilon\coloneqq f_\epsilon^{\frac{n-2}{n-1}},\qquad \epsilon>0.
}
Using the zero mode equation and testing the pointwise Schr\"odinger--Lichnerowicz formula \eqref{eq:S-L} against $f_\epsilon^{-\frac{2}{n-1}}\varphi$, we obtain
\eq{\label{eq1:regularization}
    \int f_\epsilon^{-\frac{2}{n-1}}\rRe\<i\rd A^\flat\cdot\varphi,\varphi\> = \int \rRe\<\nabla^A\varphi,\nabla^A(f_\epsilon^{-\frac{2}{n-1}}\varphi)\> + \frac{n(n-1)}{4}\int f^2f_\epsilon^{-\frac{2}{n-1}}.
}
Since
\eq{
    \nabla^A_X(f_\epsilon^{-\frac{2}{n-1}}\varphi) = -\frac{2}{n-1}f_\epsilon^{-\frac{2n}{n-1}}fX(f)\varphi + f_\epsilon^{-\frac{2}{n-1}}\nabla^A_X\varphi,
}
we have
\eq{
    \rRe\<\nabla^A\varphi,\nabla^A(f_\epsilon^{-\frac{2}{n-1}}\varphi)\> &= -\frac{2}{n-1}f_\epsilon^{-\frac{2n}{n-1}}f^2\abs{\nabla f}^2 + f_\epsilon^{-\frac{2}{n-1}}\abs{\nabla^A\varphi}^2 \\
    &= -\frac{2}{n-1}f_\epsilon^{-\frac{2n}{n-1}}f^2\abs{\nabla f}^2 + f_\epsilon^{-\frac{2}{n-1}}\Big(\frac{n}{n-1}\abs{\nabla f}^2+f^2\abs{S}^2\Big) \\
    &= \Big(\frac{n-2}{n-1}f_\epsilon^{-\frac{2}{n-1}} + \frac{2}{n-1}\epsilon^2f_\epsilon^{-\frac{2n}{n-1}} \Big)\abs{\nabla f}^2 + f_\epsilon^{-\frac{2}{n-1}}f^2\abs{S}^2.
}
where we used \eqref{eq:6}. Combining with
\eq{
    \rRe\<i\rd A^\flat\cdot\varphi,\varphi\> = f^2\<\rd A^\flat,\beta\>,
}
we see that \eqref{eq1:regularization} becomes
\eq{\label{eq2:regularization}
    \int f_\epsilon^{-\frac{2}{n-1}}f^2\<\rd A^\flat,\beta\> &= \int \Big(\frac{n-2}{n-1}f_\epsilon^{-\frac{2}{n-1}} + \frac{2}{n-1}\epsilon^2f_\epsilon^{-\frac{2n}{n-1}} \Big)\abs{\nabla f}^2 \\
    &\quad + \int f_\epsilon^{-\frac{2}{n-1}}f^2\abs{S}^2 + \frac{n(n-1)}{4}\int f_\epsilon^{-\frac{2}{n-1}}f^2.
}

First, by the pointwise bound $\abs{\beta}^2\leq\big[\tfrac{n}{2}\big]$,
\eq{
    f_\epsilon^{-\frac{2}{n-1}}f^2\<\rd A^\flat,\beta\> \leq f^{-\frac{2}{n-1}}f^2\cdot\abs{\rd A^\flat}\cdot\Big[\frac{n}{2}\Big]^{\frac 12} = \Big[\frac{n}{2}\Big]^{\frac 12} u^2\abs{\rd A^\flat}.
}
Therefore,
\eq{
    \int f_\epsilon^{-\frac{2}{n-1}}f^2\<\rd A^\flat,\beta\> \leq \Big[\frac{n}{2}\Big]^{\frac 12} \norm{\rd A^\flat}_{\frac{n}{2}}\norm{u}_{\frac{2n}{n-2}}^2.
}
Since each term on the right-hand side of \eqref{eq2:regularization} is nonnegative, Fatou's lemma implies
\eq{\label{eq3:regularization}
    \int_{\{f>0\}} f^{-\frac{2}{n-1}}\abs{\nabla f}^2 < \infty, \qquad \int_{\{f>0\}} u^2\abs{S}^2 < \infty.
}
Note that
\eq{
    \nabla u_\epsilon = \frac{n-2}{n-1}f_\epsilon^{-\frac{n}{n-1}}f\nabla f,
}
hence
\eq{\label{eq4:regularization}
    \abs{\nabla u_\epsilon}^2 \leq \frac{(n-2)^2}{(n-1)^2}f^{-\frac{2}{n-1}}\abs{\nabla f}^2 \qquad\hbox{a.e. on}\quad \{f>0\}.
}
Since $u_\epsilon\to u$ uniformly as $\epsilon\to0$, by \eqref{eq3:regularization} and \eqref{eq4:regularization}, we know that $u_\epsilon\to u$ strongly in $W^{1,2}(\S^n)$, and
\eq{\label{eq5:regularization}
    \abs{\nabla u}^2 = \frac{(n-2)^2}{(n-1)^2}f^{-\frac{2}{n-1}}\abs{\nabla f}^2 \qquad\hbox{a.e. on}\quad \{f>0\}.
}

Second, since
\eq{
    \frac{\epsilon^2f_\epsilon^{-\frac{2n}{n-1}}}{f^{-\frac{2}{n-1}}} = \frac{(f/\epsilon)^{\frac{2}{n-1}}}{\big(1+(f/\epsilon)^2\big)^{\frac{n}{n-1}}} \leq 1 \qquad\hbox{on}\quad \{f>0\},
}
and it converges to $0$ pointwise on $\{f>0\}$, using \eqref{eq3:regularization} and dominated convergence, we have
\eq{
    \int \epsilon^2f_\epsilon^{-\frac{2n}{n-1}}\abs{\nabla f}^2 \to 0.
}

Third, by dominated (and monotone) convergence, the remaining terms in \eqref{eq2:regularization} converge as $\epsilon\to0$. Passing to the limit yields
\eq{
    \int u^2 \<\rd A^\flat,\beta\> = \frac{n-1}{n-2}\int\abs{\nabla u}^2 + \frac{n(n-1)}{4}\int u^2 + \int u^2\abs{S}^2,
}
which is precisely \eqref{eq:158}.

Finally, a standard cut-off argument justifies integrations by parts near $\{f=0\}$. Note also that the presence of a zero set does not affect the equality case, since equality forces $f\equiv{\rm const}>0$.

\subsection{Case \texorpdfstring{$n=3$}{n=3}}
For convenience of the reader and for completeness, we very briefly sketch the proof for $n=3$ in \cite{WZ26curl}. For details, see \cite{WZ26curl}.

\

\noindent\textit{Sketch of the proof for $n=3$.} The goal is to prove the sharp bound
\eq{\label{eq:FL-sharp-sketch}
    \|\curl A\|_{\frac32}\ge 3\,\omega_3^{\frac23}
}
whenever there exists a nontrivial solution $\varphi\not\equiv0$ of $\D\varphi=iA\cdot\varphi$ on $\S^3$. The existence of its minimizer was proved in \cite{FL1}.
Write $f\coloneqq|\varphi|$ and $\phi\coloneqq f^{-1}\varphi$ on $\{f>0\}$, so $|\phi|\equiv1$.
A dimension--$3$ spinorial algebra identity yields a decomposition of $\nabla^A\varphi$ into orthogonal pieces and the pointwise identity
\eq{\label{eq:Kato-id-sketch}
    |\nabla^A\varphi|^2=\frac32|\nabla f|^2+f^2|S|^2,
}
where $S$ is a trace-free symmetric tensor (so the remainder term $f^2|S|^2$ is nonnegative but cannot be discarded without losing sharpness).
Next, define the unit vector field $\xi\coloneqq\rRe\langle i\phi,e_j\cdot\phi\rangle e_j$ (so $\xi\cdot\phi=i\phi$) and the $1$-form $\alpha\coloneqq f\,\xi^\flat$.
A direct computation relates $S$ to $\rd\alpha$ and gives
\eq{\label{eq:S-vs-dalpha-sketch}
    f^2|S|^2\ge \frac{3}{8}|\rd\alpha|^2.
}

Integrating the Schr\"odinger--Lichnerowicz formula for $\D^A$ against $\varphi$ and using \eqref{eq:Kato-id-sketch}, one derives an inequality of the form
\eq{\label{eq:main-ineq-sketch}
    \int_{\S^3} f\,|\curl A|\ge \frac32\,\omega_3^{\frac23}\|f\|_3+\int_{\S^3} f\,|S|^2.
}
By H\"older,
$\int f|\curl A|\le \|f\|_3\,\|\curl A\|_{3/2}$, so it remains to show that the remainder term contributes at the same scale:
\eq{\label{eq:claim-sketch}
    \|f\|_3^{-1}\int_{\S^3} f\,|S|^2\ge \frac32\,\omega_3^{\frac23}.
}
To prove \eqref{eq:claim-sketch}, consider the conformal metric $\tilde g\coloneqq f^2g_{\rm st}$ and apply the sharp conformal lower bound for the first positive curl eigenvalue in $[g_{\rm st}]$; this eigenvalue bound is equivalent to (and can be proved from) the sharp curl--Sobolev inequality (Theorem 1.1 in \cite{WZ26curl}).
Using the Rayleigh quotient for the curl operator with test form $\alpha$ (and the fact that $\rd^*(f\alpha)=0$, which follows from the zero mode equation), one obtains a lower bound for $\int f^{-1}|\rd\alpha|^2$ in terms of $\|f\|_3$.
Combining this with \eqref{eq:S-vs-dalpha-sketch} yields \eqref{eq:claim-sketch}, and inserting \eqref{eq:claim-sketch} into \eqref{eq:main-ineq-sketch} gives \eqref{eq:FL-sharp-sketch}.
The equality case follows by tracing back the equalities in the Cauchy--Schwarz, Sobolev, and curl-eigenvalue estimates, forcing $f$ to be constant (modulo conformal transformations) and hence $(\varphi,A)$ to arise from a Killing spinor and its associated Reeb field (modulo gauge and conformal transformations).

\section{Appendix. The Clifford algebra on \texorpdfstring{$\R^4$}{R4}}
\label{appendix_Clifford}

Let $\phi$ be a unit chiral spinor on $\R^4$. Since $\Sigma\R^4\cong\mathbb{C}^4$, without loss of generality we may assume
\eq{
    \phi = \phi_1 = (1,0,0,0)^T.
}
Define Pauli matrices by
\eq{
    \sigma_1\coloneqq\begin{pmatrix} 0 & 1 \\ 1 & 0 \end{pmatrix}, \qquad \sigma_2\coloneqq\begin{pmatrix} 0 & -i \\ i & 0 \end{pmatrix}, \qquad \sigma_3\coloneqq\begin{pmatrix} 1 & 0 \\ 0 & -1 \end{pmatrix},
}
and set
\eq{
    e_1\cdot \coloneqq \begin{pmatrix} 0 & \sigma_1 \\ -\sigma_1 & 0 \end{pmatrix}, \quad e_2\cdot \coloneqq \begin{pmatrix} 0 & \sigma_2 \\ -\sigma_2 & 0 \end{pmatrix}, \quad e_3\cdot \coloneqq \begin{pmatrix} 0 & \sigma_3 \\ -\sigma_3 & 0 \end{pmatrix}, \quad e_4\cdot \coloneqq \begin{pmatrix} 0 & iI_2 \\ iI_2 & 0 \end{pmatrix},
}
where $I_2$ is the $2\times2$ identity matrix. 

By direct computation, one checks that
\eq{
    e_j\cdot e_k\cdot + \, e_k\cdot e_j\cdot = -2\delta_{jk}, \quad 1\leq j,k\leq 4.
}
Hence $\{e_1,e_2,e_3,e_4\}$ forms a representation of the Clifford algebra. If we denote
\eq{
    \phi_2\coloneqq (0,1,0,0)^T, \qquad \phi_3\coloneqq (0,0,1,0)^T, \qquad \phi_4\coloneqq (0,0,0,1)^T,
}
then $\{\phi_1,\phi_2,\phi_3,\phi_4\}$ forms a complex orthonormal basis of $\Sigma\R^4$. Direct computation shows that
\eq{\label{eq:B1}
    e_1\cdot\phi_1 = -\phi_4, \quad e_2\cdot\phi_1 = -i\phi_4, \quad e_3\cdot\phi_1 = -\phi_3, \quad e_4\cdot\phi_1 = i\phi_3,
}
and
\eq{\label{eq:B2}
    e_1\cdot\phi_2 = -\phi_3, \quad e_2\cdot\phi_2 = i\phi_3, \quad e_3\cdot\phi_2 = \phi_4, \quad e_4\cdot\phi_2 = i\phi_4.
}
Moreover,
\eq{\label{eq:B3}
    e_1\cdot e_2\cdot\phi_1 &= -i\phi_1, \quad e_1\cdot e_3\cdot\phi_1 = -\phi_2, \quad e_1\cdot e_4\cdot\phi_1 = i\phi_2, \\
    e_2\cdot e_3\cdot\phi_1 &= -i\phi_2, \quad e_2\cdot e_4\cdot\phi_1 = -\phi_2, \quad e_3\cdot e_4\cdot\phi_1 = i\phi_1.
}
By \eqref{eq:B3} we see that $\phi=\phi_1$ is a negative spinor, since
\eq{
    e_1\cdot e_2\cdot e_3\cdot e_4\cdot\phi = \phi.
}

\part{The improved Sobolev inequality under barycenter constraint}\label{part2}

This part is devoted to a new improved Sobolev inequality under barycenter constraint. It could be viewed as an independent paper.

\section{Aubin's improved Sobolev inequality}
\label{sec:aubin-improved-sobolev}

In this part, we first recall the following optimal Sobolev inequality, due to Aubin \cite{Aubin79} and Talenti \cite{Talenti76}.

\begin{theorem}\label{thm:sharp-sobolev-aubin-talenti}
Let $n\ge 3$ and $2^*:=\frac{2n}{n-2}$. For every $u\in \dot{W}^{1,2}(\R^n)$,
\eq{\label{eq:sharp-sobolev-Rn}
\Big(\int_{\R^n} |u|^{2^*}\,\rd x\Big)^{\frac{n-2}{n}}\le S_n^{-1}\int_{\R^n} |\nabla u|^2\,\rd x,
}
where $S_n=\frac{n(n-2)}{4}\,\omega_n^{\frac{2}{n}}$ is optimal (here $\omega_n=|\mathbb S^n|$).
Moreover, equality holds if and only if $u\equiv 0$ or
\[
 u(x)=c\,\Big(\frac{\lambda}{1+\lambda^2|x-x_0|^2}\Big)^{\frac{n-2}{2}}
\]
for some constants $c\in\R\setminus\{0\}$, $\lambda>0$ and $x_0\in\R^n$.
\end{theorem}

It is well known that this inequality is equivalent to the following one on $\S^n$.
\begin{theorem}\label{thm:sharp-sobolev-sphere}
Let $n\ge 3$ and $2^*:=\frac{2n}{n-2}$. For every $v\in W^{1,2}(\mathbb S^n)$,
\eq{\label{eq:sharp-sobolev-Sn}
\Big(\int_{\mathbb S^n} |v|^{2^*}\,\rd\sigma\Big)^{\frac{n-2}{n}}
\le S_n^{-1}\int_{\mathbb S^n}\Big(|\nabla_{\mathbb S^n} v|^2+\frac{n(n-2)}{4}\,v^2\Big)\,\rd\sigma,
}
where $\rd\sigma$ is the standard volume element on $\mathbb S^n$ and $S_n=\frac{n(n-2)}{4}\,\omega_n^{\frac{2}{n}}$ is sharp.
Moreover, equality holds if and only if $v\equiv 0$ or $v$ is (up to a nonzero multiplicative constant) the pullback of a constant function by a conformal diffeomorphism of $\mathbb S^n$; equivalently,
\[
 v(\xi)=c\,\left(\frac{1-|a|^2}{|\xi-a|^2}\right)^{\frac{n-2}{2}},\qquad \xi\in\mathbb S^n\subset\R^{n+1},
\]
for some $c\in\R\setminus\{0\}$ and some $a\in\R^{n+1}$ with $|a|<1$.
\end{theorem}

If one is interested in a smaller class of functions, it is natural to hope that we have an ``improved Sobolev inequality'', i.e., an inequality with a better constant. The first such an improved Sobolev inequality related to our work is the following

\begin{theorem}\label{cor:aubin-barycenter}
Let $n\ge 3$. There exists a constant $c>0$ such that for any $u\in W^{1,2}(\S^n)$ with
\eq{\label{eq:barycenter-constraint-sphere}
    \int_{\S^n} x\,|u|^{\frac{2n}{n-2}}\,\rdV = 0 \in \R^{n+1},
}
we have
\eq{\label{eq:aubin-improved-sobolev-barycenter}
    \int_{\S^n}\Big(\abs{\nabla u}^2 + \frac{n(n-2)}{4}u^2\Big)\,\rdV
    \;\geq\; \frac{n(n-2)}{4}\omega_n^{\frac{2}{n}}\norm{u}_{\frac{2n}{n-2}}^2
    + c\Big(\omega_n^{\frac{2}{n}}\norm{u}_{\frac{2n}{n-2}}^2 - \int_{\S^n} u^2\,\rdV\Big).
}
\end{theorem}

This result is usually attributed to Aubin \cite{Aubin79}.  A detailed  proof can be found in \cite{ChangYang91}, Li \cite{LiYanyan96} and Dolbeault--Esteban--Loss \cite{DEL17}. A similar result on the fourth order Paneitz operator was proved by Djadli--Malchiodi--Ould Ahmedou \cite{Malchiodi02}.

Since the work of Aubin, there is a lot of work on improved Sobolev inequalities. Here we merely mention the result of Hang--Wang \cite{HangWang22} on improved Sobolev inequality under higher order moment constraints.

In this part, we are interested in improved Sobolev inequalities in  form \eqref{eq:aubin-improved-sobolev-barycenter}. Specifically, we are interested in the question of what is the best value of $\mathfrak{a}_n$. Using a sequence of even functions with two blow-ups, it is not difficult to show that 
\eq{\label{conj0}
    \mathfrak{a}_n \leq \frac{n(n-2)}{4}(2^{\frac{2}{n}}-1) \eqcolon c_n^*.
}
We conjecture that \eq{\label{conj} \mathfrak{a}_n =c^*_n.}
In fact, Dolbeault conjectured  in his survey \cite[equation (4.3)]{Dolbeault21}  that \eqref{conj} is true in the class of even functions $u$, i.e., $u(-x)=u(x)$, $\forall $ $x\in \S^n$.

Our Theorem~\ref{thm_a_n0} gives a universal  lower bound of $\mathfrak{a}_n$ for all $n\geq5$. For convenience, let us state it again.

\begin{theorem}\label{partII_main_thm} We have
\eq{
\mathfrak{a}_n > \frac{n(n-2)}{4(n^2-3n+1)}=:c_n.
}
    \end{theorem}
    As showed in Part 1, the lower estimate is enough to our aim. It is certainly not optimal.  Since $c_3>c_3^*$, Theorem \ref{thm_a_n0} is not true for $n=3$.

For comparison, in dimension $2$ the critical Sobolev embedding is replaced by the Moser--Trudinger inequality, and several improved versions under the barycenter-type constraints have been studied since Chang--Yang \cite{ChangYang91}, which ususally are called Chang-Yang conjecture. The original Chang-Yang conjecture was proved finally by Gui-Moradifam in \cite{Gui_Moradim}. 
Very recently the higher dimensional generalization of the Chang--Yang conjecture was resolved in  \cite{GLWY26}. For related development and history we refer to \cite{Gui_Moradim} and \cite{GLWY26}
and the references therein.

\section{Proof of Theorem~\ref{thm_a_n0}}\label{sec8}

In this section, we prove Theorem~\ref{partII_main_thm}, 
under the assumption of the existence of minimizers of the functional corresponding to \eqref{eq:aubin-improved-sobolev-barycenter} for a given $\mathfrak{a}_n$,   if $\mathfrak{a}_n\le c_n$.
The existence of a minimizer will be proved in Section~\ref{sec9} below.

Let $u$ be a minimizer. 
It satisfies the  Euler--Lagrange equation 
\eq{
    -\Delta u + \Big(\frac{n(n-2)}{4}+\mathfrak{a}_n\Big)u - \Big(\frac{n(n-2)}{4}+\mathfrak{a}_n\Big)u^{\frac{n+2}{n-2}} = (\mu\cdot x)u^{\frac{n+2}{n-2}},
}
for some constant vector $\mu\in\R^{n+1}$, or equivalently
\eq{\label{eq:important}
    -\Delta u + \Big(\frac{n(n-2)}{4}+\mathfrak{a}_n\Big)u = \Big(\frac{n(n-2)}{4}+\mathfrak{a}_n + \mu\cdot x\Big)u^{\frac{n+2}{n-2}},
}
with constraints
\eq{\label{eq0}
    \int u^{\frac{2n}{n-2}} = \omega_n, \qquad \int xu^{\frac{2n}{n-2}}=0.
}
As a minimizer, we may assume $u\ge 0$ and hence $u>0$ by the maximum principle.

The second variation is
\eq{
    Q(v,v) = \int \Big\{ \abs{\nabla v}^2 + \Big(\frac{n(n-2)}{4}+\mathfrak{a}_n\Big)v^2 \Big\} - \frac{n+2}{n-2}\int \Big(\frac{n(n-2)}{4}+\mathfrak{a}_n + \mu\cdot x\Big)u^{\frac{4}{n-2}}v^2
}
for  any variational vector $v$, which satisfies 
\eq{\label{eq:constraints_v}
    \int u^{\frac{n+2}{n-2}}v=0, \qquad \int xu^{\frac{n+2}{n-2}}v=0.
}
Denote by $\mathcal{C}$ the space of admissible variations $v$ such that \eqref{eq:constraints_v} holds.
Then it suffices to show that
\eq{
    Q \geq 0 \ \hbox{ on }\ \mathcal{C} \implies \mathfrak{a}_n > \frac{n(n-2)}{4(n^2-3n+1)},
}
since $u$ is a minimizer.
From now on we assume by contradiction that 
\eq{\label{eq-1}
    \mathfrak{a}_n \leq \frac{n(n-2)}{4(n^2-3n+1)}.
}
We first construct an admissible test function $v\in\mathcal{C}$ and then show that $Q(v,v)<0$, and hence get a contradiction to \eqref{eq-1}.

\subsection{The test function}

In this subsection, we construct an admissible test function in the spirit of our previous work \cite{WZ25} and \cite{WaZh26}. The idea is to utilize the conformal Killing directions. 

Let $\ell$ be a linear function on $\S^n$, i.e. a first-order spherical harmonic,
\eq{
    \ell = e\cdot x
}
for some unit vector $e\in\R^{n+1}$. Then
\eq{
    \nabla \ell = \nabla(e\cdot x) = e^\top = e - \ell x
}
is a conformal Killing field. Let $\phi_t$ be the corresponding one-parameter family of conformal diffeomorphisms, characterized by
$\frac{\rd}{\rd t}\phi_t = \nabla \ell$.

In view of the normalization
\eq{
    \int_{\S^n} u^{\frac{2n}{n-2}} = \omega_n, \qquad\hbox{or equivalently}\quad \fint_{\S^n} u^{\frac{2n}{n-2}} = 1,
}
we consider the conformal orbit
\eq{
     u_t \coloneqq J_{\phi_t}^{\frac{n-2}{2n}} \cdot u\circ\phi_t,
}
where $J_{\phi_t}$ denotes the Jacobian of $\phi_t$. This choice preserves the normalization:
\eq{
     \fint u_t^{\frac{2n}{n-2}} = \fint J_{\phi_t}\cdot( u\circ\phi_t)^{\frac{2n}{n-2}} = \fint u^{\frac{2n}{n-2}} = 1.
}
By definition of $\ell$, we have
\eq{
    -\Delta \ell = n\ell, \qquad \nabla^2\ell = -\ell g_{\S^n}.
}
It follows
\eq{
    \frac{\rd}{\rd t}\Big|_{t=0} u_t &= \frac{\rd}{\rd t}\Big|_{t=0} J_{\phi_t}^{\frac{n-2}{2n}} \cdot u\circ\phi_t \\
    &= \frac{n-2}{2n}u\cdot\frac{\rd}{\rd t}\Big|_{t=0} J_{\phi_t} + \frac{\rd}{\rd t}\Big|_{t=0} u\circ\phi_t \\
    &= \frac{n-2}{2n} u {\rm div}(\nabla \ell) + \<\nabla u,\nabla \ell\> \\
    &= \<\nabla u,\nabla \ell\> - \frac{n-2}{2}\ell u \eqcolon w.
}

Recall that $u$ be a minimizer of \eqref{eq:a_n0}. The variation $w=\frac{\rd}{\rd t}|_{t=0}u_t$ is the natural infinitesimal deformation along the conformal orbit, and we would like to use it as a test function.
However, to enforce the second admissibility constraint in  \eqref{eq:constraints_v}, we need to modify $w$ by adding an appropriate multiple of $\ell u$, where $\ell$ is the first-order harmonic:
\eq{\label{eq:v}
    v = w + q\ell u,
}
where $q\in\R$ will be chosen below.
We now use the two constraints
\eq{
    \fint u^{\frac{n+2}{n-2}}v = 0, \qquad \fint xu^{\frac{n+2}{n-2}}v = 0
}
to determine suitable $e$ and $q$.

Note that
\eq{
    \fint u^{\frac{n+2}{n-2}}\ell u = \fint \ell u^{\frac{2n}{n-2}} = 0,
}
and integrating by parts gives
\eq{
    \fint u^{\frac{n+2}{n-2}}\<\nabla u,\nabla\ell\> = \frac{n-2}{2n}\fint (-\Delta\ell)u^{\frac{2n}{n-2}} = \frac{n-2}{2} \fint \ell u^{\frac{2n}{n-2}} = 0.
}
It follows
\eq{\label{eq:v1}
    \fint u^{\frac{n+2}{n-2}}w = 0, \qquad \fint u^{\frac{n+2}{n-2}}v = 0.
}
Thus the first constraint holds for all choices of $e$ and $q$. For the second constraint, using
\eq{
    {\rm div}(x\nabla\ell) = e - (n+1)\ell x,
}
we see that
\eq{
    \fint xu^{\frac{n+2}{n-2}}\<\nabla u,\nabla\ell\> &= -\frac{n-2}{2n}\fint (e - (n+1)\ell x)\, u^{\frac{2n}{n-2}},
}
and
\eq{
    \fint xu^{\frac{n+2}{n-2}}\ell u = \fint \ell xu^{\frac{2n}{n-2}}.
}
Therefore
\eq{\label{eq1}
    \fint xu^{\frac{n+2}{n-2}}v = \fint \Big(-\frac{n-2}{2n}e+\big(\frac{n-2}{2n}+q\big)\ell x\Big)u^{\frac{2n}{n-2}}.
}
Set an $(n+1)\times(n+1)$ matrix
\eq{
    M \coloneqq \fint x\otimes x \, u^{\frac{2n}{n-2}},
}
then
\eq{
    Me = \fint (e\cdot x) x u^{\frac{2n}{n-2}} = \fint \ell x u^{\frac{2n}{n-2}}.
}
Inserting above equations into \eqref{eq1} gives
\eq{
    \fint xu^{\frac{n+2}{n-2}}v = -\frac{n-2}{2n}e  + \big(\frac{n-2}{2n}+q\big)Me,
}
where we used \eqref{eq0}. Hence
\eq{
    \fint xu^{\frac{n+2}{n-2}}v = 0 \quad\iff\quad Me = \frac{\frac{n-2}{2n}}{\frac{n-2}{2n}+q}e \eqcolon \lambda e.
}
Namely, $v$ is admissible if and only if $e$ is a $\lambda$-eigenvector of $M$, and
\eq{\label{eq_q}
    q = \frac{n-2}{2n}\frac{1-\lambda}{\lambda}.
}
For the later estimates, we choose $\lambda$ to be the largest eigenvalue of $M$. Since ${\rm tr}(M) = 1$, we have
\eq{
    \lambda \geq \frac{1}{n+1},
}
and
\eq{\label{eq-3}
    \fint \ell x\,u^{\frac{2n}{n-2}} = \lambda e, \qquad \fint \ell^2 u^{\frac{2n}{n-2}} = \lambda.
}

It remains to show that for $v$ defined by \eqref{eq:v} one has
\eq{
    Q(v,v)<0.
}

\subsection{Key estimates}

We begin with a key quantitative bound on the largest eigenvalue $\lambda$ of the moment matrix $M$ defined above. We remark that a similar result was recently proved by Gui--Li--Wei--Ye \cite{GLWY26} in their resolution of the higher dimensional Chang-Yang conjecture.

\begin{lemma}\label{lem2.1}
Let $n\geq5$. Then
\eq{
    \frac{1}{n+1}\leq \lambda < \frac{3}{10}.
}
As a consequence,
\eq{\label{eq-2.5}
    q = \frac{n-2}{2n}\frac{1-\lambda}{\lambda} > \frac{7}{10}.
}
\end{lemma}
\begin{proof}
Note that $\abs{\ell}\leq1$. Define
\eq{
    f_\pm \coloneqq u\sqrt{\frac{1\pm\ell}{2}}.
}
Then
\eq{
    \abs{f_+}^2 + \abs{f_-}^2 = u^2,
}
and
\eq{
    \nabla f_\pm = \sqrt{\frac{1\pm\ell}{2}}\nabla u \pm \frac{u\nabla\ell}{2\sqrt{2}\sqrt{1\pm\ell}}. 
}
Since
\eq{
    \abs{\nabla\ell}^2 = \abs{\nabla(e\cdot x)}^2 = \abs{e - \ell x}^2 = 1-\ell^2,
}
we have
\eq{
    \abs{\nabla f_+}^2 + \abs{\nabla f_-}^2 = \abs{\nabla u}^2 + \frac{1}{4}u^2.
}
We now apply the critical Sobolev inequality
\eq{
    \int \abs{\nabla f}^2 + \frac{n(n-2)}{4}\int f^2 \geq \frac{n(n-2)}{4}\omega_n^{\frac{2}{n}}\Big(\int f^{\frac{2n}{n-2}}\Big)^{\frac{n-2}{n}}
}
to $f=f_\pm$ and sum the resulting inequalities. This yields
\eq{\label{eq2}
    \int\abs{\nabla u}^2 + \frac{1}{4}\int u^2 + \frac{n(n-2)}{4}\int u^2 \geq \frac{n(n-2)}{4}\omega_n^{\frac{2}{n}}\Big\{\Big(\int f_+^{\frac{2n}{n-2}}\Big)^{\frac{n-2}{n}} + \Big(\int f_-^{\frac{2n}{n-2}}\Big)^{\frac{n-2}{n}}\Big\}.
}
Set
\eq{\label{eq:k}
    \kappa \coloneqq \frac{n(n-2)}{4}+\mathfrak{a}_n.
}
Recall from \eqref{eq:a_n0} that
\eq{\label{eq3}
    \int\abs{\nabla u}^2 = \kappa\Big\{ \omega_n^{\frac{2}{n}}\Big(\int u^{\frac{2n}{n-2}}\Big)^{\frac{n-2}{n}} - \int u^2 \Big\} = \kappa\Big(\omega_n - \int u^2\Big).
}
Moreover, our assumption \eqref{eq-1} implies
\eq{\label{eq4}
    \frac{4\mathfrak{a}_n}{n(n-2)} \leq \frac{1}{n^2-3n+1}.
}
Combining \eqref{eq2}, \eqref{eq3}, and \eqref{eq4} gives
\eq{\label{eq5}
    &\Bigg( \fint \Big(\frac{1+\ell}{2}\Big)^{\frac{n}{n-2}}u^{\frac{2n}{n-2}} \Bigg)^{\frac{n-2}{n}} + \Bigg( \fint \Big(\frac{1-\ell}{2}\Big)^{\frac{n}{n-2}}u^{\frac{2n}{n-2}} \Bigg)^{\frac{n-2}{n}} \\
    &\leq \frac{4}{n(n-2)}\fint\abs{\nabla u}^2 + \frac{1}{n(n-2)}\fint u^2 + \fint u^2 \\
    &= \frac{4}{n(n-2)}\kappa\Big(1-\fint u^2\Big) + \frac{1}{n(n-2)}\fint u^2 + \fint u^2 \\
    &= 1 + \frac{4\mathfrak{a}_n}{n(n-2)}\Big(1-\fint u^2\Big) + \frac{1}{n(n-2)}\fint u^2 \\
    &\leq 1 + \frac{1}{n^2-3n+1}\Big(1-\fint u^2\Big) + \frac{1}{n^2-3n+1}\fint u^2 \\
    &= 1 + \frac{1}{n^2-3n+1}.
}
We then estimate the left-hand side of \eqref{eq5} from below. Set $r\coloneqq\frac{n}{n-2}\in(1,\frac{5}{3}]$. Then
\eq{\label{eq6}
    (1+x)^r \geq 1 + rx + \frac{r(r-1)}{2}x^2 + \frac{r(r-1)(r-2)}{6}x^3, \qquad \abs{x}\leq1.
}
In order to apply \eqref{eq6} to $x=\pm\ell$, we note that
\eq{\label{eq7}
    \fint \ell\,u^{\frac{2n}{n-2}}=0, \qquad \fint \ell^2\,u^{\frac{2n}{n-2}} = \lambda, \qquad\hbox{and set}\quad \tau\coloneqq \fint \ell^3\,u^{\frac{2n}{n-2}},
}
where we used $Me=\lambda e$ in the second identity. Combining \eqref{eq6} and \eqref{eq7} gives
\eq{\label{eq8}
    &\Bigg( \fint \Big(\frac{1+\ell}{2}\Big)^{\frac{n}{n-2}}u^{\frac{2n}{n-2}} \Bigg)^{\frac{n-2}{n}} + \Bigg( \fint \Big(\frac{1-\ell}{2}\Big)^{\frac{n}{n-2}}u^{\frac{2n}{n-2}} \Bigg)^{\frac{n-2}{n}} \\
    &\geq \frac{1}{2}\Big(1+\frac{r(r-1)}{2}\lambda+\frac{r(r-1)(r-2)}{6}\tau\Big)^{\frac{n-2}{n}} + \frac{1}{2}\Big(1+\frac{r(r-1)}{2}\lambda-\frac{r(r-1)(r-2)}{6}\tau\Big)^{\frac{n-2}{n}}.
}
Note that for $C>0$,
\eq{\label{eq9}
    x \mapsto (C-x)^{\frac{n-2}{n}} + (C+x)^{\frac{n-2}{n}} \quad\hbox{is non-increasing for }0\leq x\leq C.
}
Since $\abs{\ell}\leq1$, we have
\eq{
    \abs{\tau} = \Abs{\fint \ell^3\,u^{\frac{2n}{n-2}}} \leq \fint \ell^2\,u^{\frac{2n}{n-2}} = \lambda.
}
Suppose $\lambda\geq\frac{3}{10}$. Then \eqref{eq8} implies that
\eq{
    &\Bigg( \fint \Big(\frac{1+\ell}{2}\Big)^{\frac{n}{n-2}}u^{\frac{2n}{n-2}} \Bigg)^{\frac{n-2}{n}} + \Bigg( \fint \Big(\frac{1-\ell}{2}\Big)^{\frac{n}{n-2}}u^{\frac{2n}{n-2}} \Bigg)^{\frac{n-2}{n}} \\
    &\geq \frac{1}{2}\Big(1+\frac{r(r-1)(r+1)}{6}\lambda\Big)^{\frac{n-2}{n}} + \frac{1}{2}\Big(1+\frac{r(r-1)(5-r)}{6}\lambda\Big)^{\frac{n-2}{n}} \\
    &\geq \frac{1}{2}\Big(1+\frac{r(r-1)(r+1)}{20}\Big)^{\frac{n-2}{n}} + \frac{1}{2}\Big(1+\frac{r(r-1)(5-r)}{20}\Big)^{\frac{n-2}{n}},
}
where we used \eqref{eq9} for $\tau$ in the first step, and $\lambda\geq\frac{3}{10}$ in the second step. Finally, note that
\eq{\label{eq10}
    (1+x)^{\frac{n-2}{n}} &\geq 1 + \frac{n-2}{n}x - \frac{\frac{n-2}{n}\big(1-\frac{n-2}{n}\big)}{2}x^2 \\
    &= 1 + \frac{n-2}{n}x - \frac{n-2}{n^2}x^2 \qquad\hbox{for}\quad x\geq0.
}
Also, we have
\eq{
    \frac{r(r-1)(r+1)}{20} + \frac{r(r-1)(5-r)}{20} = \frac{3r(r-1)}{10},
}
and
\eq{
    \Big(\frac{r(r-1)(r+1)}{20}\Big)^2 + \Big(\frac{r(r-1)(5-r)}{20}\Big)^2 = \frac{r^2(r-1)^2}{400}\Big((r+1)^2+(5-r)^2\Big) \leq \frac{r^2(r-1)^2}{20},
}
where we used $r\in(1,\frac{5}{3}]\implies (r+1)^2+(5-r)^2\leq20$ in the last step. Hence \eqref{eq10} implies
\eq{
    &\frac{1}{2}\Big(1+\frac{r(r-1)(r+1)}{20}\Big)^{\frac{n-2}{n}} + \frac{1}{2}\Big(1+\frac{r(r-1)(5-r)}{20}\Big)^{\frac{n-2}{n}} \\
    &\geq 1 + \frac{n-2}{n}\cdot \frac{3r(r-1)}{20} - \frac{n-2}{n^2}\cdot\frac{r^2(r-1)^2}{40} \\
    &= 1 + \frac{3}{10(n-2)} - \frac{1}{10(n-2)^3} > 1 + \frac{1}{n^2-3n+1},
}
which contradicts \eqref{eq5}.
Hence $\lambda<\frac{3}{10}$, and we complete the proof.
\end{proof}

We next use the Euler--Lagrange equation \eqref{eq:important}, written as, 
\eq{\label{eq11}
    -\Delta u + \kappa u = (\kappa+\mu\cdot x)u^{\frac{n+2}{n-2}},
}
to derive a Kazdan--Warner type identity for the Lagrange multiplier $\mu$.

\begin{lemma}\label{lem7.2}
We have
\eq{\label{eq17}
    \mathfrak{a}_n\fint xu^2 = -\frac{n-2}{2n}(\rid-M)\mu.
}
\end{lemma}
\begin{proof}
Set $X_j\coloneqq\nabla x_j$. We test \eqref{eq11} against $X_j(u)=\<\nabla u,\nabla x_j\>$. We start with
\eq{\label{eq12}
    \fint (-\Delta u)X_j(u) &= \fint \nabla_{\nabla u}\nabla_{X_j}u = \fint (\nabla^2u)(\nabla u,X_j) + \fint \<\nabla_{\nabla u}\nabla x_j,\nabla u\> \\
    &= \frac{1}{2}\fint \nabla_{X_j}(\abs{\nabla u}^2) - \fint x_j\abs{\nabla u}^2 = -\frac{1}{2}\fint {\rm div}(\nabla x_j)\abs{\nabla u}^2  - \fint x_j\abs{\nabla u}^2 \\
    &= \frac{n-2}{2}\fint x_j\abs{\nabla u}^2,
}
where we used $\nabla^2x_j = -x_jg_{\S^n}$. Second,
\eq{\label{eq13}
    \fint uX_j(u) = \frac{1}{2}\fint X_j(u^2) = -\frac{1}{2}\fint {\rm div}(\nabla x_j)u^2 = \frac{n}{2}\fint x_ju^2.
}
Third,
\eq{\label{eq14}
    \fint (\kappa+\mu\cdot x)u^{\frac{n+2}{n-2}}X_j(u) &= \frac{n-2}{2n}\fint\<\nabla(u^{\frac{2n}{n-2}}),(\kappa+\mu\cdot x)X_j\> \\
    &= -\frac{n-2}{2n}\fint \Big(-nx_j(\kappa+\mu\cdot x)+X_j(\mu\cdot x)\Big)u^{\frac{2n}{n-2}} \\
    &= \frac{n-2}{2}\fint(\kappa+\mu\cdot x)x_ju^{\frac{2n}{n-2}} - \frac{n-2}{2n}\fint X_j(\mu\cdot x)u^{\frac{2n}{n-2}}.
}
Combining \eqref{eq11}, \eqref{eq12}, \eqref{eq13}, and \eqref{eq14} gives
\eq{\label{eq15}
    \frac{n-2}{2}\fint x_j\abs{\nabla u}^2 +  \frac{n}{2}\kappa\fint x_ju^2 = \frac{n-2}{2}\fint(\kappa+\mu\cdot x)x_ju^{\frac{2n}{n-2}} - \frac{n-2}{2n}\fint X_j(\mu\cdot x)u^{\frac{2n}{n-2}}.
}
We then test \eqref{eq11} against $x_ju$. Since
\eq{
    \fint (-\Delta u)x_ju = \fint \<\nabla u,\nabla(x_ju)\> = \fint x_j\abs{\nabla u}^2 + \frac{n}{2}\fint x_ju^2,
}
we have
\eq{\label{eq16}
    \fint x_j\abs{\nabla u}^2 + \Big(\frac{n}{2}+\kappa\Big)\fint x_ju^2 = \fint(\kappa+\mu\cdot x)x_ju^{\frac{2n}{n-2}}.
}
Combining \eqref{eq15} and \eqref{eq16} gives
\eq{
    \Big\{\frac{n}{2}\kappa - \frac{n-2}{2}\Big(\frac{n}{2}+\kappa\Big)\Big\} \fint x_ju^2 = -\frac{n-2}{2n}\fint X_j(\mu\cdot x)u^{\frac{2n}{n-2}}.
}
Since
\eq{
    \frac{n}{2}\kappa - \frac{n-2}{2}\Big(\frac{n}{2}+\kappa\Big) = \frac{n}{2}\Big(\frac{n(n-2)}{4}+\mathfrak{a}_n\Big) - \frac{n-2}{2}\Big(\frac{n}{2}+\frac{n(n-2)}{4}+\mathfrak{a}_n\Big) = \mathfrak{a}_n,
}
and
\eq{
    X_j(\mu\cdot x) = \<e_j-(e_j\cdot x)x,\mu-(\mu\cdot x)x\> = e_j\cdot\mu - (e_j\cdot x)(\mu\cdot x),
}
we have
\eq{
    \mathfrak{a}_n\fint x_ju^2 = -\frac{n-2}{2n}\Big( e_j\cdot\mu - \fint (e_j\cdot x)(\mu\cdot x)u^{\frac{2n}{n-2}}\Big).
}
The claim follows.
\end{proof}
The Kazdan-Warner identity 
\eqref{eq17} implies $\mathfrak{a}_n>0$, which is a known result as mentioned   in the previous section.

\begin{corollary}
$\mathfrak{a}_n>0$.
\end{corollary}
\begin{proof}
By \eqref{eq17}, $\mathfrak{a}_n=0$ implies $M\mu=\mu$. But by Lemma~\ref{lem2.1}, the maximal eigenvalue of $M$ satisfies $\lambda<\frac{3}{10}$. Hence $\mu=0$. As a consequence, the Euler--Lagrange equation \eqref{eq:important} becomes the Yamabe equation, hence $u$ is constant modulo conformal transformations, which contradicts to Theorem \ref{thm:attainment} below.
\end{proof}

Therefore, it suffices to consider $\mathfrak{a}_n>0$.

\begin{lemma}\label{lem7.4}
We have
\eq{\label{eq23}
    \Big\{ \Big( \frac{n(n-2)}{8\mathfrak{a}_n} + \frac{n-2}{2n} \Big)(1-\lambda) + \lambda \Big\}\abs{\mu} \leq \kappa\Big( 1 - \fint u^2 \Big).
}
\end{lemma}
\begin{proof}
It suffices to consider $\mu\neq0$. Set
\eq{\label{eq22}
    s\coloneqq \abs{\mu}^{-2}\fint (\mu\cdot x)^2u^{\frac{2n}{n-2}} = \abs{\mu}^{-2}\<M\mu,\mu\> \leq \lambda.
}
Testing \eqref{eq17} against $\mu$ gives
\eq{\label{eq18}
    \mathfrak{a}_n\fint(\mu\cdot x)u^2 = -\frac{n-2}{2n}\fint \big(\abs{\mu}^2-(\mu\cdot x)^2\big)u^{\frac{2n}{n-2}} = -\frac{n-2}{2n}\abs{\mu}^2(1-s).
}
Since $\mathfrak{a}_n>0$ and $s\leq\lambda<\frac{3}{10}$, it follows
\eq{
    \fint(\mu\cdot x)u^2<0.
}
Note that \eqref{eq16} implies
\eq{
    \fint x\abs{\nabla u}^2 + \Big(\frac{n}{2}+\kappa\Big)\fint xu^2 = \fint (\mu\cdot x)x\,u^{\frac{2n}{n-2}}.
}
Testing against $\mu$ gives
\eq{\label{eq19}
    \fint (\mu\cdot x)\abs{\nabla u}^2 + \Big(\frac{n}{2}+\kappa\Big)\fint (\mu\cdot x)u^2 = \fint (\mu\cdot x)^2\,u^{\frac{2n}{n-2}} = \abs{\mu}^2s.
}
By Cauchy--Schwarz we have
\eq{\label{eq20}
    \fint (\mu\cdot x)\abs{\nabla u}^2 \leq \abs{\mu}\fint \abs{\nabla u}^2 = \kappa\abs{\mu}\Big(1-\fint u^2\Big),
}
where we used \eqref{eq3}. Combining \eqref{eq18}, \eqref{eq19}, and \eqref{eq20} gives
\eq{
    \abs{\mu}^2s \leq \kappa\abs{\mu}\Big(1-\fint u^2\Big) - \Big(\frac{n}{2}+\kappa\Big)\frac{n-2}{2n\mathfrak{a}_n}\abs{\mu}^2(1-s),
}
or equivalently
\eq{\label{eq21}
    \Big\{ \Big(\frac{n(n-2)}{8\mathfrak{a}_n}+\frac{n-2}{2n}\Big)(1-s)+s \Big\}\abs{\mu} \leq \kappa\Big( 1 - \fint u^2 \Big).
}
Note that
\eq{
    \mathfrak{a}_n\leq \frac{n(n-2)}{4(n^2-3n+1)} \implies \frac{n(n-2)}{8\mathfrak{a}_n}>1,
}
hence the left-hand side of \eqref{eq21} is decreasing in $s$. Then \eqref{eq22} implies
\eq{
    \Big(\frac{n(n-2)}{8\mathfrak{a}_n}+\frac{n-2}{2n}\Big)(1-s)+s \geq \Big(\frac{n(n-2)}{8\mathfrak{a}_n}+\frac{n-2}{2n}\Big)(1-\lambda)+\lambda.
}
The claim follows.
\end{proof}

\begin{lemma}
We have
\eq{\label{eq23.5}
    \Big\{ \Big( \frac{n(n-2)}{8\mathfrak{a}_n} + \frac{n-2}{2n} \Big)(1-\lambda) + \lambda \Big\}\Abs{\fint (\mu\cdot x)\ell^2\,u^{\frac{2n}{n-2}}} \leq \lambda\sqrt{1-\lambda}\,\kappa\Big( 1 - \fint u^2 \Big).
}
\end{lemma}
\begin{proof}
Note that
\eq{
    \fint (\mu\cdot x)\ell^2\,u^{\frac{2n}{n-2}} = \fint (\mu\cdot x)(\ell^2-\lambda)u^{\frac{2n}{n-2}}.
}
By H\"older's inequality we have
\eq{\label{eq24}
    \Big( \fint (\mu\cdot x)\ell^2\,u^{\frac{2n}{n-2}} \Big)^2 &\leq \fint (\mu\cdot x)^2 u^{\frac{2n}{n-2}}\cdot\fint (\ell^2-\lambda)^2 u^{\frac{2n}{n-2}} = \<M\mu,\mu\> \cdot \fint (\ell^4-2\lambda\ell^2+\lambda^2) u^{\frac{2n}{n-2}} \\
    &\leq \lambda\abs{\mu}^2\cdot\fint (\ell^2-2\lambda\ell^2+\lambda^2) u^{\frac{2n}{n-2}} \leq \lambda\abs{\mu}^2(\lambda-2\lambda^2+\lambda^2) = \lambda^2(1-\lambda)\abs{\mu}^2,
}
where we used $\abs{\ell}\leq1$ and \eqref{eq-3}. The claim then follows from \eqref{eq23} and \eqref{eq24}.
\end{proof}

\subsection{An explicit formula for \texorpdfstring{$Q(v,v)$}{Q(v,v)}}

Recall that the second variation can be written as
\eq{
    Q(v,v) &= \int \Big\{ \abs{\nabla v}^2 + \Big(\frac{n(n-2)}{4}+\mathfrak{a}_n\Big)v^2 \Big\} - \frac{n+2}{n-2}\int \Big(\frac{n(n-2)}{4}+\mathfrak{a}_n + \mu\cdot x\Big)u^{\frac{4}{n-2}}v^2 \\
    &\eqcolon \int vLv,
}
where the stability operator is
\eq{
    L = -\Delta + \kappa - \frac{n+2}{n-2}(\kappa+\mu\cdot x)u^{\frac{4}{n-2}}.
}

\begin{lemma}
We have
\eq{\label{eq27.0}
    L(\ell u) = 2\ell u-2w-\frac{4}{n-2}(\kappa+\mu\cdot x)\ell\,u^{\frac{n+2}{n-2}},
}
and
\eq{\label{eq27}
    Lw = 2\mathfrak{a}_n\ell u + \big( \mu\cdot e - (\mu\cdot x)\ell \big)u^{\frac{n+2}{n-2}}.
}
\end{lemma}
\begin{proof}
First, using the Euler--Lagrange equation \eqref{eq11} we have
\eq{\label{eq25}
    L(\ell u) &= -\Delta(\ell u) + \kappa\ell u - \frac{n+2}{n-2}(\kappa+\mu\cdot x)u^{\frac{n+2}{n-2}}\ell \\
    &= n\ell u - 2\<\nabla u,\nabla\ell\> + \Big( -\Delta u + \kappa u - \frac{n+2}{n-2}(\kappa+\mu\cdot x)u^{\frac{n+2}{n-2}} \Big)\ell \\
    &= n\ell u - 2\Big( w + \frac{n-2}{2}\ell u\Big) - \frac{4}{n-2}(\kappa+\mu\cdot x)u^{\frac{n+2}{n-2}}\ell \\
    &= 2\ell u-2w-\frac{4}{n-2}(\kappa+\mu\cdot x)\ell\,u^{\frac{n+2}{n-2}}. 
}
Second, the Bochner formula on $\S^n$ implies
\eq{
    -\Delta\<\nabla u,\nabla\ell\> &= -\<\nabla\Delta u,\nabla\ell\> - \<\nabla u,\nabla\Delta\ell\> - 2\<\nabla^2u,\nabla^2\ell\> - 2(n-1)\<\nabla u,\nabla\ell\> \\
    &= -\<\nabla\Delta u,\nabla\ell\> + n\<\nabla u,\nabla\ell\> + 2\ell\Delta u - 2(n-1)\<\nabla u,\nabla\ell\> \\
    &= -\<\nabla\Delta u,\nabla\ell\> - (n-2)\<\nabla u,\nabla\ell\> + 2\ell\Delta u.
}
Using the Euler--Lagrange equation \eqref{eq11}, we have
\eq{
    &-\Delta\<\nabla u,\nabla\ell\> = \<\nabla( -\kappa u + (\kappa+\mu\cdot x)u^{\frac{n+2}{n-2}} ),\nabla\ell\> - (n-2)\<\nabla u,\nabla\ell\> - 2\ell( -\kappa u + (\kappa+\mu\cdot x)u^{\frac{n+2}{n-2}} ) \\
    &= \Big(-(\kappa+n-2)+\frac{n+2}{n-2}(\kappa+\mu\cdot x)u^{\frac{4}{n-2}}\Big)\<\nabla u,\nabla\ell\> + u^{\frac{n+2}{n-2}}\<\nabla(\mu\cdot x),\nabla\ell\> + 2\kappa\ell u - 2\ell(\kappa+\mu\cdot x)u^{\frac{n+2}{n-2}}
}
Hence
\eq{\label{eq26}
    L\<\nabla u,\nabla\ell\> &= -\Delta\<\nabla u,\nabla\ell\> + \kappa\<\nabla u,\nabla\ell\> - \frac{n+2}{n-2}(\kappa+\mu\cdot x)u^{\frac{4}{n-2}}\<\nabla u,\nabla\ell\> \\
    &= -(n-2)\<\nabla u,\nabla\ell\> + u^{\frac{n+2}{n-2}}\<\nabla(\mu\cdot x),\nabla\ell\> + 2\kappa\ell u - 2\ell(\kappa+\mu\cdot x)u^{\frac{n+2}{n-2}}.
}
Combining \eqref{eq25} and \eqref{eq26}, and using $\<\nabla(\mu\cdot x),\nabla\ell\>=\mu\cdot e - (\mu\cdot x)\ell$, we obtain \eqref{eq27}.    
\end{proof}

\begin{lemma}
We have
\eq{\label{eq28}
    \fint \ell uw = -\frac{1}{2}\fint u^2 + \frac{3}{2}\fint \ell^2u^2,
}
and
\eq{\label{eq29}
    \fint (\mu\cdot x)\ell u^{\frac{n+2}{n-2}}w = \frac{n-2}{n} \fint (\mu\cdot x)\ell^2\,u^{\frac{2n}{n-2}}.
}
\end{lemma}
\begin{proof}
First,
\eq{
    \fint \ell uw &= \fint \ell u\Big(\<\nabla u,\nabla\ell\> - \frac{n-2}{2}\ell u\Big) = -\frac{1}{4}\fint u^2\Delta(\ell^2) - \frac{n-2}{2}\fint \ell^2u^2 \\
    &= -\frac{1}{2}\fint u^2\Big( \ell\Delta\ell + \abs{\nabla\ell}^2 \Big) - \frac{n-2}{2}\fint \ell^2u^2 \\
    &= -\frac{1}{2}\fint u^2\Big( -n\ell^2 + 1 - \ell^2 \Big) - \frac{n-2}{2}\fint \ell^2u^2 \\
    &= -\frac{1}{2}\fint u^2 + \frac{3}{2}\fint \ell^2u^2.
}
Second, using
\eq{
    {\rm div}(u^{\frac{2n}{n-2}}\nabla\ell) = \frac{2n}{n-2}u^{\frac{n+2}{n-2}}\<\nabla u,\nabla\ell\> - nu^{\frac{2n}{n-2}}\ell = \frac{2n}{n-2}u^{\frac{n+2}{n-2}}w,
}
we have
\eq{
    \fint (\mu\cdot x)\ell u^{\frac{n+2}{n-2}}w &= \frac{n-2}{2n}\fint (\mu\cdot x)\ell\,{\rm div}(u^{\frac{2n}{n-2}}\nabla\ell) \\
    &= -\frac{n-2}{2n}\fint \Big( (\mu\cdot x)\abs{\nabla\ell}^2 + \ell\<\nabla(\mu\cdot x),\nabla\ell\>\Big)u^{\frac{2n}{n-2}} \\
    &= -\frac{n-2}{2n}\fint \Big( (\mu\cdot x)(1-\ell^2) + \ell(\mu\cdot e - (\mu\cdot x)\ell)\Big)u^{\frac{2n}{n-2}} \\
    &= \frac{n-2}{n} \fint (\mu\cdot x)\ell^2\,u^{\frac{2n}{n-2}}.
}
The claim follows.
\end{proof}

\begin{proposition}
We have
\eq{\label{eq33}
    \frac{1}{\omega_n}Q(v,v) &= - \frac{4\kappa}{n-2}\lambda q^2 + (q^2-\mathfrak{a}_n)\fint u^2 + \big(\mathfrak{a}_n(3+4q)-q^2\big)\fint \ell^2u^2 \\
    &\quad -\Big(\frac{n-2}{n}+2q+\frac{4q^2}{n-2}\Big) \fint (\mu\cdot x)\ell^2\,u^{\frac{2n}{n-2}}.
}
\end{proposition}
\begin{proof}
First, using \eqref{eq:v1}, \eqref{eq27}, \eqref{eq28}, and \eqref{eq29} we have
\eq{\label{eq30}
    \frac{1}{\omega_n}Q(w,w) &= \fint wLw = \fint w\Big(2\mathfrak{a}_n\ell u + \big( \mu\cdot e - (\mu\cdot x)\ell \big)u^{\frac{n+2}{n-2}}\Big) \\
    &= 2\mathfrak{a}_n\fint \ell uw + (\mu\cdot e)\fint w\,u^{\frac{n+2}{n-2}} - \fint (\mu\cdot x)\ell\,u^{\frac{n+2}{n-2}}w \\
    &= -\mathfrak{a}_n \fint u^2 + 3\mathfrak{a}_n\fint \ell^2u^2 - \frac{n-2}{n} \fint (\mu\cdot x)\ell^2\,u^{\frac{2n}{n-2}}.
}
Second, using \eqref{eq27} we have
\eq{\label{eq31}
    \frac{1}{\omega_n}Q(w,\ell u) &= \fint \ell uLw = \fint \ell u\Big( 2\mathfrak{a}_n\ell u + \big( \mu\cdot e - (\mu\cdot x)\ell \big)u^{\frac{n+2}{n-2}} \Big) \\
    &= 2\mathfrak{a}_n\fint \ell^2u^2 - \fint (\mu\cdot x)\ell^2\,u^{\frac{2n}{n-2}}.
}
Third, using \eqref{eq-3}, \eqref{eq27.0}, \eqref{eq28}, and \eqref{eq29} we have
\eq{\label{eq32}
    \frac{1}{\omega_n}Q(\ell u,\ell u) &= \fint \ell uL(\ell u) = \fint \ell u\Big( 2\ell u-2w-\frac{4}{n-2}(\kappa+\mu\cdot x)\ell\,u^{\frac{n+2}{n-2}} \Big) \\
    &= 2\fint \ell^2u^2 - 2\fint \ell uw - \frac{4\kappa}{n-2}\fint \ell^2\,u^{\frac{2n}{n-2}} - \frac{4}{n-2}\fint (\mu\cdot x)\ell^2\,u^{\frac{2n}{n-2}} \\
    &= \fint u^2 - \fint \ell^2u^2 - \Big(n+\frac{4\mathfrak{a}_n}{n-2}\Big)\lambda - \frac{4}{n-2}\fint (\mu\cdot x)\ell^2\,u^{\frac{2n}{n-2}}.
}
The claim now follows from \eqref{eq30}, \eqref{eq31}, and \eqref{eq32}.
\end{proof}

For simplicity, we denote
\eq{
    I_1 \coloneqq \fint u^2, \qquad I_2 \coloneqq \fint \ell^2u^2, \qquad I_3 \coloneqq \fint (\mu\cdot x)\ell^2\,u^{\frac{2n}{n-2}}, \qquad I_4 \coloneqq \fint \ell^2\abs{\nabla u}^2,
}
and
\eq{
    C_1 \coloneqq n+1+\kappa, \qquad C_2 \coloneqq \mathfrak{a}_n(3+4q)-q^2, \qquad C_3 \coloneqq \frac{n-2}{n} + 2q + \frac{4q^2}{n-2}.
}
Then \eqref{eq33} becomes
\eq{\label{eq33.1}
    \frac{1}{\omega_n}Q(v,v) = - \frac{4\kappa}{n-2}\lambda q^2 + (q^2-\mathfrak{a}_n)I_1 + C_2I_2 - C_3I_3.
}

\subsection{Estimates of \texorpdfstring{$Q(v,v)$}{Q(v,v)}}\label{sec7.4}

Testing the Euler--Lagrange equation \eqref{eq11} against $\ell^2u$ gives
\eq{\label{eq34}
    \fint\<\nabla u,\nabla(\ell^2u)\> + \kappa I_2 = \kappa\lambda + I_3.
}
Using \eqref{eq28} we have
\eq{\label{eq35}
    \fint\<\nabla u,\nabla(\ell^2u)\> &= \fint \ell^2\abs{\nabla u}^2 + 2\fint \ell u\<\nabla u,\nabla\ell\> \\
    &= \fint \ell^2\abs{\nabla u}^2 + 2\fint \ell u \Big( w + \frac{n-2}{2}\ell u\Big) \\
    &= I_4 - I_1 +(n+1)I_2.
}
Combining \eqref{eq34} and \eqref{eq35} gives
\eq{
    I_2 = \frac{1}{C_1}(\kappa\lambda+I_1+I_3-I_4).
}
Hence \eqref{eq33.1} becomes
\eq{
    \frac{1}{\omega_n}Q(v,v) = \lambda\Big(\frac{\kappa C_2}{C_1} -\frac{4\kappa}{n-2}q^2\Big) + \Big(q^2-\mathfrak{a}_n+\frac{C_2}{C_1}\Big)I_1 - \Big(C_3-\frac{C_2}{C_1} \Big)I_3 - \frac{C_2}{C_1}I_4.
}

\begin{lemma}
We have
\eq{\label{eq37}
    -\Big(C_3-\frac{C_2}{C_1}\Big)I_3 \leq \Big(C_3-\frac{C_2}{C_1}\Big)C_4(1-I_1),
}
where
\eq{
    C_4 \coloneqq \frac{ \kappa\lambda\sqrt{1-\lambda} }{ \big( \frac{n(n-2)}{8\mathfrak{a}_n} + \frac{n-2}{2n} \big)(1-\lambda) + \lambda }.
}
\end{lemma}
\begin{proof}
We first claim
\eq{\label{eq36}
    C_3-\frac{C_2}{C_1} > 0.
}
If $C_2\leq0$, the claim is trivial. If $C_2>0$, then using \eqref{eq-2.5} we have
\eq{
    \frac{C_1}{n-2} = \frac{n+1+\kappa}{n-2} \geq \frac{n+1}{n-2} + \frac{n}{4} \geq \frac{13}{4} > \frac{\mathfrak{a}_n(3+4q)-q^2}{q^2} = \frac{C_2}{q^2},
}
hence
\eq{\label{eq36.5}
    \frac{C_2}{C_1} < \frac{q^2}{n-2} < C_3.
}
The conclusion then follows from \eqref{eq23.5} and \eqref{eq36}.
\end{proof}

\begin{lemma}
Set 
$C_2^-\coloneqq \max\{-C_2,0\}$. Then
\eq{\label{eq38}
    -\frac{C_2}{C_1}I_4 \leq \frac{\kappa C_2^-}{C_1}(1-I_1).
}
\end{lemma}
\begin{proof}
Since $\abs{\ell}\leq1$, using \eqref{eq3} we have
\eq{
    I_4 \leq \fint \abs{\nabla u}^2 = \kappa(1-I_1).
}
The conclusion follows.
\end{proof}

Combining \eqref{eq33.1}, \eqref{eq37}, and \eqref{eq38} gives
\eq{\label{eq:rewrite}
    \frac{1}{\omega_n}Q(v,v) &\leq \Big\{ q^2-\mathfrak{a}_n-\frac{4\kappa}{n-2}\lambda q^2 + \frac{C_2}{C_1}(1+\kappa\lambda) \Big\}I_1 \\
    &\quad + \Big\{ -\frac{4\kappa}{n-2}\lambda q^2 + \frac{\kappa}{C_1}(\lambda C_2+C_2^-)+\Big(C_3-\frac{C_2}{C_1}\Big)C_4 \Big\}(1-I_1).
}
We now estimate the two terms separately.

\bigskip

\noindent\textbf{Estimate of the first term.} By definition of $q$, \eqref{eq_q}, we have 
\eq{
    \lambda = \frac{n-2}{2nq+n-2}.
}
It follows
\eq{\label{eq38.5}
    &q^2-\mathfrak{a}_n-\frac{4\kappa}{n-2}\lambda q^2 + \frac{C_2}{C_1}(1+\kappa\lambda) \\
    &= \frac{ (n-2-2q)\{ 2\mathfrak{a}_n(q+1)((n-2)^2+4\mathfrak{a}_n-2nq) -n^2(n+2)q^2 \} }{ (2nq+n-2)(4\mathfrak{a}_n+n^2+2n+4) }.
}
First, since $\lambda\geq\frac{1}{n+1}$, we have
\eq{
    q = \frac{n-2}{2n}\Big(\frac{1}{\lambda}-1\Big) \leq \frac{n-2}{2},
}
hence
\eq{\label{eq39}
    n-2-2q \geq 0.
}
Second, since
\eq{
    \mathfrak{a}_n \le  \frac{n(n-2)}{4(n^2-3n+1)} < \frac{1}{2} < \frac{(n-2)^2}{4},
}
we have
\eq{\label{eq40}
    2\mathfrak{a}_n(q+1)((n-2)^2+4\mathfrak{a}_n-2nq) < \frac{n}{2}\cdot 2(n-2)^2 = n(n-2)^2.
}
Third, since $\lambda<\frac{3}{10}$, we have
\eq{
    q = \frac{n-2}{2n}\Big(\frac{1}{\lambda}-1\Big) > \frac{7(n-2)}{6n} > \frac{n-2}{n},
}
hence
\eq{\label{eq41}
    n^2(n+2)q^2 > n^2(n+2)\cdot\frac{(n-2)^2}{n^2} = (n+2)(n-2)^2 > n(n-2)^2.
}
Combining \eqref{eq38.5}, \eqref{eq39}, \eqref{eq40}, and \eqref{eq41} gives
\eq{
    q^2-\mathfrak{a}_n-\frac{4\kappa}{n-2}\lambda q^2 + \frac{C_2}{C_1}(1+\kappa\lambda) \leq 0.
}

\bigskip

\noindent\textbf{Estimate of the second term when $C_2<0$.} 
In this case, we have
\eq{
    C_2^-=-C_2 = -\mathfrak{a}_n(3+4q)+q^2 \leq q^2 \leq \frac{(n-2)^2}{4}.
}
First, since $\lambda\geq\frac{1}{n+1}$, we have $\frac{1-\lambda}{\lambda}\leq n$, hence
\eq{
    \frac{\kappa}{C_1}(\lambda C_2+C_2^-) \leq \frac{-C_2}{C_1}\kappa(1-\lambda) \leq \frac{\kappa q^2}{C_1}(1-\lambda) \leq \frac{n}{C_1}\kappa\lambda q^2 = \Big( 1 - \frac{n+1+\mathfrak{a}_n}{C_1}\Big)\frac{4\kappa}{n-2}\lambda q^2.
}
Note that
\eq{
    5n(n+1+\mathfrak{a}_n) - 12C_1 = 2n^2-n-12+(5n-12)\mathfrak{a}_n > 0,
}
hence
\eq{\label{eq42}
    \frac{\kappa}{C_1}(\lambda C_2+C_2^-) < \Big(1-\frac{12}{5n}\Big)\frac{4\kappa}{n-2}\lambda q^2.
}
Second, since $q\leq\frac{n-2}{2}$, we have
\eq{
    C_3 = \frac{n-2}{n} + 2q + \frac{4q^2}{n-2} < 1 + 2q + 2q = 1 + 4q,
}
and
\eq{\label{eq43}
    -\frac{C_2}{C_1} = \frac{q^2-\mathfrak{a}_n(3+4q)}{n+1+\frac{n(n-2)}{4}+\mathfrak{a}_n} < \frac{q^2}{\big(\frac{n-2}{2}\big)^2} \leq 1.
}
Moreover,
\eq{\label{eq44}
    C_4 = \frac{ \kappa\lambda\sqrt{1-\lambda} }{ \big( \frac{n(n-2)}{8\mathfrak{a}_n} + \frac{n-2}{2n} \big)(1-\lambda) + \lambda } < \frac{ \kappa\lambda\sqrt{1-\lambda} }{ \frac{n(n-2)}{8\mathfrak{a}_n}(1-\lambda) } < \frac{8\mathfrak{a}_n\kappa\lambda}{n(n-2)\sqrt{1-3/10}} < \frac{48\mathfrak{a}_n\kappa\lambda}{5n(n-2)}.
}
Combining \eqref{eq42}, \eqref{eq43}, and \eqref{eq44} gives
\eq{\label{eq45}
    \Big(C_3 - \frac{C_2}{C_1}\Big)C_4 < \frac{48\mathfrak{a}_n(3+4q)\kappa\lambda}{5n(n-2)} < \frac{48q^2\kappa\lambda}{5n(n-2)} = \frac{12}{5n}\cdot\frac{4\kappa}{n-2}\lambda q^2,
}
where we used $C_2=\mathfrak{a}_n(3+4q)-q^2<0$. By \eqref{eq42} and \eqref{eq45} we have
\eq{
    -\frac{4\kappa}{n-2}\lambda q^2 + \frac{\kappa}{C_1}(\lambda C_2+C_2^-)+\Big(C_3-\frac{C_2}{C_1}\Big)C_4 < 0.
}

\bigskip

\noindent\textbf{Estimate of the second term when $C_2\geq0$.}
In this case, we have $C_2^-=0$. First, since
\eq{
    \frac{n(n-2)}{8\mathfrak{a}_n} + \frac{n-2}{2n} \geq \frac{n(n-2)}{8}\cdot\frac{4(n^2-3n+1)}{n(n-2)} + \frac{n-2}{2n} > 1,
}
we have
\eq{
    C_5 \coloneqq \Big( \frac{n(n-2)}{8\mathfrak{a}_n} + \frac{n-2}{2n} \Big)(1-\lambda) + \lambda > 1 > \sqrt{1-\lambda}.
}
Recall \eqref{eq36.5} that
\eq{
    \frac{C_2}{C_1} < \frac{q^2}{n-2}.
}
Hence
\eq{\label{eq46}
    &\frac{\kappa}{C_1}(\lambda C_2+C_2^-)+\Big(C_3-\frac{C_2}{C_1}\Big)C_4
    = \frac{C_2}{C_1}\kappa\lambda + \Big(C_3-\frac{C_2}{C_1}\Big)\frac{\kappa\lambda\sqrt{1-\lambda}}{C_5} \\
    &= \frac{C_3}{C_5}\kappa\lambda\sqrt{1-\lambda} + \frac{C_2}{C_1}\kappa\lambda\Big(1-\frac{\sqrt{1-\lambda}}{C_5}\Big) < \frac{C_3}{C_5}\kappa\lambda\sqrt{1-\lambda} + \frac{q^2}{n-2}\kappa\lambda\Big(1-\frac{\sqrt{1-\lambda}}{C_5}\Big) \\
    &= \frac{\kappa\lambda q^2}{n-2} + \frac{\sqrt{1-\lambda}}{C_5}\Big(\frac{n-2}{q^2}C_3-1\Big)\frac{\kappa\lambda q^2}{n-2}.
}
Note that
\eq{
    \frac{n-2}{q^2}C_3-1 &= \frac{n-2}{q^2}\Big(\frac{n-2}{n} + 2q + \frac{4q^2}{n-2}\Big)-1 = \frac{(n-2)^2}{nq^2} + \frac{2(n-2)}{q} + 3 \\
    &= \frac{4n\lambda^2}{(1-\lambda)^2} + \frac{4n\lambda}{1-\lambda} + 3 = \frac{3+(4n-6)\lambda+3\lambda^2}{(1-\lambda)^2},
}
and
\eq{
    C_5 &= \Big( \frac{n(n-2)}{8\mathfrak{a}_n} + \frac{n-2}{2n} \Big)(1-\lambda) + \lambda \geq \Big( \frac{n^2-3n+1}{2} + \frac{n-2}{2n} \Big)(1-\lambda) + \lambda.
}
Hence
\eq{\label{eq47}
    \frac{\sqrt{1-\lambda}}{C_5}\Big(\frac{n-2}{q^2}C_3-1\Big) \leq \frac{3+(4n-6)\lambda+3\lambda^2}{(1-\lambda)^{3/2}\big\{\big( \frac{n^2}{2}-\frac{3n}{2}+1-\frac{1}{n}\big)(1-\lambda) + \lambda\big\}} \eqcolon f_n(\lambda).
}
Since
\eq{
    \frac{n^2}{2}-\frac{3n}{2}+1-\frac{1}{n} > 1,
}
we see that $f_n(\lambda)$ is increasing in $\lambda$. Moreover, one checks that $f_n(\lambda)$ is decreasing in $n$. Hence
\eq{\label{eq48}
    f_n(\lambda) \leq f_5\Big(\frac{3}{10}\Big) < 3.
}
Combining \eqref{eq46}, \eqref{eq47}, and \eqref{eq48} gives
\eq{
    -\frac{4\kappa}{n-2}\lambda q^2 + \frac{\kappa}{C_1}(\lambda C_2+C_2^-)+\Big(C_3-\frac{C_2}{C_1}\Big)C_4 < 0.
}
Hence we complete the proof.

\section{Blow-up analysis}\label{sec9}

In this section we prove that the infimum \eqref{eq:a_n0} is attainable. For simplicity, we set
\eq{
    p \coloneqq \frac{2n}{n-2}.
}
By homogeneity and Kato's inequality, we may assume $u\geq0$ and normalize
\eq{\label{eq:1000}
    \fint_{\S^n}u^p\,\rdV=1,\qquad\fint_{\S^n}xu^p\,\rdV=0.
}
Then \eqref{eq:a_n0} becomes
\eq{\label{eq:critical-quotient}
    \mathfrak{a}_n=\inf\Big\{\frac{\fint_{\S^n}|\nabla u|^2\,\rdV}{1-\fint_{\S^n}u^2\,\rdV}-\frac{n(n-2)}4 \,\Big|\, \fint u^p=1,\ \fint xu^p=0,\ \fint u^2<1
    \Big\}.
}

\begin{theorem}\label{thm:attainment}
Let $n\ge5$. If
\eq{
      \mathfrak{a}_n\leq \frac{n(n-2)}{4(n^2-3n+1)} \eqcolon c_n,
}
then the infimum in \eqref{eq:critical-quotient} is attained by some non-constant $0<u\in C^\infty(\S^n)$ satisfying \eqref{eq:1000}.
\end{theorem}

The proof is built on the standard concentration--compactness argument. We sketch the proof for the reader's convenience.

\subsection{Subcritical minimizers}

For $2<s<p$, we set
\eq{\label{eq:subcritical-def}
    \mathfrak{a}_{n,s} \coloneqq \inf\Big\{\frac{\fint_{\S^n}|\nabla v|^2\,\rdV}{1-\fint_{\S^n}v^2\,\rdV}-\frac n{s-2} \,\Big|\, \fint v^s=1,\ \fint xv^s=0,\ \fint v^2<1 \Big\}.
}
In particular, $\mathfrak{a}_{n,p}=\mathfrak{a}_n$.

Recall Beckner's  sharp interpolation inequality  on the sphere (see \cite{Beckner})
\eq{\label{eq:beckner}
    \left(\fint_{\S^n}\abs{v}^s\,\rdV\right)^{2/s}\leq \fint_{\S^n}v^2\,\rdV +\frac{s-2}{n}\fint_{\S^n}\abs{\nabla v}^2\,\rdV,\quad 2<s\le p.
}
As a simple consequence, $\mathfrak{a}_{n,s}\geq0$.

Let $p_j\nearrow p$. The idea is to apply the compact embedding to obtain a minimizer $0<u_j\in C^\infty(\S^n)$ that attains the infimum $\mathfrak{a}_{n,p_j}$.

\begin{proposition}\label{prop8.1}
For all sufficiently large $j$, the infimum $\mathfrak{a}_{n,p_j}$ is attained by some
$u_j\in C^{\infty}(\S^n)$ with $u_j>0$ and
\eq{\label{eq:subconstraints}
    \fint_{\S^n} u_j^{p_j}\,\rdV=1,\qquad \fint_{\S^n} x u_j^{p_j}\,\rdV=0.
}
\end{proposition}

\begin{proof}
Fix $j$ and let $(v_k)$ be a minimizing sequence for $\mathfrak{a}_{n,p_j}$ subject to
\eq{\label{eq:subconstraints-vk}
    \fint v_k^{p_j}=1,\qquad \fint x v_k^{p_j}=0,\qquad \fint v_k^2<1.
}
Since $p_j<p$, the embeddings $W^{1,2}(\S^n)\hookrightarrow L^{p_j}(\S^n)$ and
$W^{1,2}(\S^n)\hookrightarrow L^2(\S^n)$ are compact. Thus, once we know that the
denominators $1-\fint v_k^2$ stay uniformly away from $0$, the direct method yields a
non-constant minimizer.
\end{proof}

We therefore exclude $1-\fint v_k^2\to 0$.

\begin{lemma}\label{lem:constant}
Suppose that $v_k$ is a minimizing sequence of $\mathfrak{a}_{n,p_j}$ satisfying
\eq{\label{eq:1001}
    \fint v_k^{p_j}=1, \qquad \fint xv_k^{p_j}=0, \qquad 1-\fint v_k^2\to0.
}
Then $v_k\to1$ strongly in $W^{1,2}(\S^n)$ and
\eq{\label{eq:constant-threshold}
    \liminf_{k\to\infty} \Big( \frac{\fint\abs{\nabla v_k}^2}{1-\fint v_k^2} -\frac n{p_j-2} \Big) \geq \frac{n+2}{p_j-2}.
}
\end{lemma}

\begin{proof} 
Since $(v_k)$ is minimizing, $\fint\abs{\nabla v_k}^2/(1-\fint v_k^2)$ is bounded.
The assumptions imply $\fint\abs{\nabla v_k}^2\to0$. By the
Poincar\'e inequality and $v_k\ge0$, we have $v_k\to1$ strongly in $W^{1,2}$. Write
$v_k=1+h_k$, where $-1\le h_k\to0$ strongly in $W^{1,2}$. Since $p_j\nearrow p$, there
exist $\delta>0$ and $C>0$ (independent of $k$) such that
\eq{\label{eq:1002}
    \Abs{(1+h_k)^{p_j}-1-p_jh_k-\frac{p_j(p_j-1)}2h_k^2} \leq C(\abs{h_k}^{2+\delta}+\abs{h_k}^p).
}
By Sobolev embedding, the integral of the right-hand side is
$o(\|h_k\|_{W^{1,2}}^2)$ since $p>2$. Using \eqref{eq:1002} and expanding the first two
constraints in \eqref{eq:1001}, we obtain
\eq{\label{eq:h-mean}
 \fint h_k = -\frac{p_j-1}{2}\fint h_k^2 + o(\|h_k\|_{W^{1,2}}^2),\qquad \fint xh_k = O(\|h_k\|_{W^{1,2}}^2).
}
Hence the $0$th and $1$st spherical harmonic components of $h_k$ are of second order.
Since the next eigenvalue of $-\Delta$ on $\S^n$ is $2(n+1)$, we have
\eq{\label{eq:second-gap}
 \fint|\nabla h_k|^2
 \ge 2(n+1)\fint h_k^2+o(\|h_k\|_{W^{1,2}}^2).
}
On the other hand, \eqref{eq:h-mean} gives
\eq{\label{eq:den-expansion}
 1-\fint v_k^2
 =(p_j-2)\fint h_k^2+o(\|h_k\|_{W^{1,2}}^2).
}
 It
follows from \eqref{eq:second-gap} and \eqref{eq:den-expansion} that
\eq{
    \frac{\fint\abs{\nabla v_k}^2}{1-\fint v_k^2} \geq\frac{2(n+1)+o(1)}{p_j-2+o(1)}.
}
This implies \eqref{eq:constant-threshold}.
\end{proof}

Returning to the minimizing sequence $(v_k)$ for $\mathfrak{a}_{n,p_j}$, note that
$n\ge 5$ and
\eq{
    \mathfrak{a}_n \leq c_n = \frac{n(n-2)}{4(n^2-3n+1)} < \frac{n^2-4}{4} = \lim_{j\to\infty}\frac{n+2}{p_j-2}.
}

\begin{lemma}
We have
\eq{\label{sup}
    \limsup_{s\nearrow p}\mathfrak{a}_{n,s}\leq \mathfrak{a}_n.
}
\end{lemma}
\begin{proof}
Given any admissible $u\geq0$ with
\eq{
    \fint u^p=1, \qquad \fint xu^p=0, \qquad \fint u^2<1,
}
set
\eq{
    M_s \coloneqq \fint x\otimes x\,u^s.
}
It is clear that $M_s\to M_p$ as $s\to p$. Since $M_p$ is positive definite, for $s$ close to $p$, $M_s$ is invertible. We now construct a test function $v_s$ as follows. Set
\eq{
    w_s \coloneqq u\Big(1-x\cdot M_s^{-1}\fint xu^s\Big)^{\frac{1}{s}}, \qquad v_s\coloneqq \Big(\fint w_s^s \Big)^{-\frac{1}{s}}w_s.
}
One easily checks that
\eq{
    \fint xw_s^s = \fint xu^s\Big(1-x\cdot M_s^{-1}\fint xu^s\Big) = \fint xu^s - \fint xu^s = 0,
}
hence
\eq{\label{eq:10000}
    \fint v_s^s = 1, \qquad \fint xv_s^s = 0. 
}
Since
\eq{
    \phi_s\coloneqq\Big(1-x\cdot M_s^{-1}\fint xu^s\Big)^{\frac{1}{s}}\to1 \qquad\hbox{in}\quad C^1,
}
we have
\eq{
    \norm{w_s-u}_{W^{1,2}} \lesssim \norm{\phi_s-1}_{L^\infty}\norm{u}_{W^{1,2}} + \norm{\nabla\phi_s}_{L^\infty}\norm{u}_{L^2} \to 0.
}
Finally, since
\eq{
    \Big(\fint w_s^s \Big)^{-\frac{1}{s}} \to 1,
}
we see that $v_s\to u$ strongly in $W^{1,2}$. Since $\fint u^2<1$, we have $\fint v_s^2<1$ for $s$ close to $p$. Together with \eqref{eq:10000}, we see that $v_s$ is an admissible test function for $\mathfrak{a}_{n,s}$.
By definition of $\mathfrak{a}_{n,s}$, we have
\eq{
    \mathfrak{a}_{n,s} \leq \frac{ \fint \abs{\nabla v_s}^2 }{ 1-\fint v_s^2 } - \frac{n}{s-2}.
}
Letting $s\to p$ gives
\eq{
    \limsup_{s\nearrow p}\mathfrak{a}_{n,s} \leq \frac{ \fint \abs{\nabla u}^2 }{ 1-\fint u^2 } - \frac{n}{p-2} = \frac{ \fint \abs{\nabla u}^2 }{ 1-\fint u^2 } - \frac{n(n-2)}{4}.
}
Then taking infimum over all admissible $u$ gives \eqref{sup}.
\end{proof}

Consequently, for all sufficiently large $j$,
\eq{\label{eq:choose-s}
    \mathfrak{a}_{n,p_j}\le \mathfrak{a}_n+o(1)\leq c_n+o(1), \qquad \mathfrak{a}_{n,p_j}<\frac{n+2}{p_j-2}.
}
If along a minimizing sequence we had $1-\fint v_k^2\to0$, Lemma \ref{lem:constant}
would force the minimizing values to be at least $(n+2)/(p_j-2)$, contradicting
\eqref{eq:choose-s}. Hence $1-\fint v_k^2$ is bounded away from $0$ along a minimizing
sequence, and the direct method produces a minimizer $u_j\ge0 $. 

Standard Lagrange multiplier arguments (for the two constraints in
\eqref{eq:subconstraints}) yield a vector $\mu_j\in\R^{n+1}$ such that $u_j$ satisfies
\eq{\label{eq:ELj}
    -\Delta u_j+\kappa_j u_j =(\kappa_j+\mu_j\cdot x)u_j^{p_j-1}
}
and
\eq{\label{eq:1003}
    \fint \abs{\nabla u_j}^2 = \kappa_j\Big(1-\fint u_j^2\Big),
}
where $\kappa_j\coloneqq \mathfrak{a}_{n,p_j}+\frac{n}{p_j-2}$. Since $p_j<p$, standard
subcritical elliptic regularity yields $u_j\in C^{\infty}(\S^n)$, and Harnack's
inequality together with \eqref{eq:subconstraints} implies $u_j>0$.

\subsection{Kazdan--Warner type argument}

We now run the same argument as in Section~\ref{sec8}. Set
\eq{\label{eq:M_j}
    M_j\coloneqq\fint_{\S^n}x\otimes x\,u_j^{p_j},
}
and let $\lambda_j$ be the largest eigenvalue of $M_j$. Replacing $p$ by $p_j$ in the proof of Lemma~\ref{lem2.1} and then letting $j\to\infty$, we have
\eq{\label{eq:1004}
    \frac{1}{n+1}\leq\limsup_{j\to\infty}\lambda_j\leq\frac{3}{10}.
}
As in the proof of Lemma~\ref{lem7.2}, we have the Kazdan--Warner type identity
\eq{\label{eq:KWj}
    \Big(\frac{n-2}{2}-\frac n{p_j}\Big)\fint_{\S^n}x\abs{\nabla u_j}^2 + \frac{n(p_j-2)}{2p_j}\mathfrak{a}_{n,p_j}\fint_{\S^n}xu_j^2 = -\frac1{p_j}({\rm id}-M_j)\mu_j.
}
Using $M_j\le\lambda_j\mathrm{id}$, \eqref{eq:1003}, and \eqref{eq:KWj}, we have
\eq{\label{eq:muj-bound}
    \abs{\mu_j}\leq\frac{p_j}{1-\lambda_j}\Bigg\{\Abs{\frac{n-2}{2}-\frac n{p_j}}\kappa_j+\frac{n(p_j-2)}{2p_j}\mathfrak{a}_{n,p_j}\Bigg\}.
}
Since the right-hand side is uniformly bounded, after passing to a subsequence we have
\eq{
    \mu_j\to\mu, \qquad \mathfrak{a}_{n,p_j}\to  a_\infty\leq c_n, \qquad \kappa_j\to\kappa:=\frac{n(n-2)}4+a_\infty.
}
Hence taking limit of \eqref{eq:muj-bound} and using \eqref{eq:1004}, we have 
\eq{
    \abs{\mu} \leq \frac{p}{1-3/10}a_\infty = \frac{20n}{7(n-2)}a_\infty.
}
Since $a_\infty\leq c_n$, it follows that
\eq{\label{eq:K-gap}
    \frac{4}{n(n-2)}(\kappa+\abs{\mu}) \leq 1+\frac1{n^2-3n+1}\Big(1+\frac{20n}{7(n-2)}\Big) < \Big(\frac{10}{3}\Big)^{\frac{2}{n}}.
}

\subsection{Exclusion of atoms}

By \eqref{eq:choose-s} and \eqref{eq:1003}, we see that $(u_j)$ is bounded in $W^{1,2}(\S^n)$.
After passing to the same subsequence, we have
\eq{
    u_j^{p_j}\,\omega_n^{-1}\rdV\overset{*}{\rightharpoonup}\nu, \qquad u_j^p\,\omega_n^{-1}\rdV\overset{*}{\rightharpoonup}\rho, \qquad \abs{\nabla u_j}^2\, \omega_n^{-1}\rdV\overset{*}{\rightharpoonup}\eta.
}
Assume $x_i$ is an atom of any of the three limiting measures. Set
\eq{
    \nu_i=\nu(\{x_i\}), \qquad \rho_i=\rho(\{x_i\}), \qquad \eta_i=\eta(\{x_i\}).
}
The Euler--Lagrange equation \eqref{eq:ELj} gives
\eq{\label{eq:atom1}
    \eta_i=(\kappa+\mu\cdot x_i)\nu_i.
}
Using Lions' concentration–compactness lemma (see Lions \cite{Lions85} or Struwe \cite{Struwe2000}), the critical Sobolev inequality implies that
\eq{\label{eq:atom2}
    \frac{n(n-2)}{4}\rho_i^{\frac{n-2}{n}} \leq \eta_i.
}
Similarly, H\"older's inequality gives
\eq{\label{eq:atom3}
    \nu_i \leq \rho_i.
}
Combining \eqref{eq:K-gap}, \eqref{eq:atom1}, \eqref{eq:atom2}, and \eqref{eq:atom3} gives
\eq{\label{eq:atom4}
    \nu_i \geq \Big(\frac{n(n-2)}{4(\kappa+\mu\cdot x_i)}\Big)^{\frac{n}{2}} \geq \Big(\frac{n(n-2)}{4(\kappa+\abs{\mu})}\Big)^{\frac{n}{2}} > \frac{3}{10}.
}
The limit of \eqref{eq:M_j} is
\eq{
    M_j \to M_\infty = \int_{\S^n} x\otimes x \rd\nu.
}
Using \eqref{eq:1004} we have
\eq{
    \nu_i \leq \nu_i + \int_{\S^n\backslash\{x_i\}}(x_i\cdot x)^2 \rd\nu = \int_{\S^n} (x_i\cdot x)^2 \rd\nu = x_i^TM_\infty x_i \leq \limsup_{j\to\infty}\lambda_j \leq \frac{3}{10},
}
which contradicts \eqref{eq:atom4}.

Therefore, all atoms are excluded, which means that $u_j$ converges to $u_0$ strongly in $W^{1,2}$. Now it is easy to see that $u_0$ achieves $\mathfrak{a}_n$.

\medskip\bigskip

\noindent{\it Acknowledgements.} We thank Rupert Frank and Michael Loss for their stimulating papers \cites{FL1, FL22, Frank_Loss_2024}, which provide both beautiful results and compelling open problems, inspiring  a series of work.  M.~Z. was partly supported by IRTG~3132 of the Deutsche Forschungsgemeinschaft. The authors acknowledge the use of ChatGPT in Subsection~\ref{sec7.4} for suggesting the decomposition of $Q$ into the two terms in \eqref{eq:rewrite} and estimating them separately; we verified the argument and simplified the resulting estimates.

\printbibliography

\end{document}